\documentclass[opre,nonblindrev]{informs3NoHeading} 

\OneAndAHalfSpacedXI 

\usepackage{endnotes}
\let\footnote=\endnote
\let\enotesize=\normalsize
\def\notesname{Endnotes}%
\def\makeenmark{$^{\theenmark}$}
\def\enoteformat{\rightskip0pt\leftskip0pt\parindent=1.75em
  \leavevmode\llap{\theenmark.\enskip}}

\usepackage{natbib}
 \bibpunct[, ]{(}{)}{,}{a}{}{,}%
 \def\bibfont{\small}%
 \def\bibsep{\smallskipamount}%
 \def\bibhang{24pt}%

\usepackage{algorithm}
\usepackage{algorithmic}
\usepackage{diagbox}
\usepackage{multirow}
\usepackage{enumerate}
\usepackage{mathtools}
\usepackage{mathrsfs}
\usepackage{enumitem}
\usepackage{dsfont}
\usepackage{float}
\usepackage{amssymb}
\usepackage{lscape}
\usepackage{graphicx}
\usepackage{amsmath}
\usepackage{bbm}
\usepackage{tikz}
\usepackage{booktabs}
\usepackage{enumitem}
\usepackage{natbib}
\usepackage{caption}
\usepackage[labelformat=simple]{subcaption}

\usetikzlibrary{calc,positioning}

\tikzset{round/.style = { rounded corners=2mm }}

\newcommand{\tr}{r}

\newcommand{\beq}{\begin{eqnarray}}
\newcommand{\eeq}{\end{eqnarray}}
\newcommand{\beqn}{\begin{eqnarray*}}
\newcommand{\eeqn}{\end{eqnarray*}}

\newtheorem{theorem}{Theorem}

\newtheorem{lemma}{Lemma}

\newtheorem{claim}{Claim}
\newtheorem{assumption}{Assumption}

\bibpunct{(}{)}{;}{a}{,}{,}

\EquationsNumberedThrough    

\begin{document}


\RUNAUTHOR{Jiang and Zhang}

\RUNTITLE{Online Resource Allocation with Stochastic Resource Consumption}

\TITLE{Online Resource Allocation with Stochastic Resource Consumption}

\ARTICLEAUTHORS{%
\AUTHOR{Jiashuo Jiang \quad Jiawei Zhang}

\AFF{\  \\
Department of Technology, Operations \& Statistics, Stern School of Business, New York University\\
}
}

\ABSTRACT{Abstract: we consider an online resource allocation problem where multiple resources, each with an individual initial capacity, are available to serve random requests arriving sequentially over multiple discrete time periods. At each time period, one request arrives and its associated reward and size are drawn independently from a known distribution that could be resource-dependent. Upon its arrival and revealing itself, an online decision has to be made on whether or not to accept the request. If accepted, another online decision should also be made to specify on assigning which resource to serve the request, then a certain amount of the resource equal to the size of the request will be consumed and a reward will be collected. The objective of the decision maker is to maximize the total collected reward subject to the capacity constraints. We develop near-optimal policies for our problem under two settings separately. When the reward distribution has a finite support, we assume that given each reward realization, the size distributions for different resources are perfectly positive correlated and we propose an adaptive threshold policy. We show that with certain regularity conditions on the size distribution, our policy enjoys an optimal $O(\log T)$ regret bound, where $T$ denotes the total time periods. When the support of the reward distribution is not necessarily finite, we develop another adaptive threshold policy with a $O(\log T)$ regret bound when both the reward and the size of each request are resource-independent and the size distribution satisfies the same conditions.
}


\maketitle

\section{Introduction}
We study an online resource allocation problem where multiple resources, each with an initial individual capacity, are available to serve a stream of requests arriving over a horizon of finite time periods. At each period, one request arrives with an associated reward and size, whose values are assumed to be random, following a given probability distribution that could be resource-dependent. Upon the arrival of a request, its reward and size are revealed. Then a decision-maker has to irrevocably decide to reject or accept it and to assign an available resource to serve it if accepted. If a request is served by a resource, a reward is collected and a certain amount of the resource equal to the size of the request is consumed. The decision maker's objective is to maximize the total collected reward without violating the resource capacity constraints.

The single-resource case of our model reduces to the online knapsack problem, which already enjoys a wide range of applications in logistics, scheduling and pricing, among others, as described in \cite{papastavrou1996dynamic}. Recent research on online knapsack problem have also been motivated by applications in online advertising (e.g. \cite{zhou2008budget}, \cite{balseiro2019learning}). In these models, the requests could arrive in an adversary order (e.g. \cite{zhou2008budget}) or in a stochastic order (e.g. \cite{papastavrou1996dynamic}, \cite{balseiro2019learning}). Our model assumes a stationary stochastic arriving order, i.e., the values of the reward and size of every request are drawn independently from a known distribution which is stationary across time.

The objective of this paper is to develop near optimal policies for our model. Throughout the paper,  the performance of our policies is compared against that of the prophet, i.e., the optimal offline decision maker.  We will use \textit{regret} to measure the performance of our policy, which is concerned about the \textit{additive} expected difference between the prophet and our policy. The formal definition of \textit{regret} are provided in Section $2$ after introducing the problem formulation and notations.

\subsection{Main Results and Contributions}
With additional assumptions, we develop policies with regret bounds of the order $\log T$, where $T$ is the total number of periods. More specifically, we assume that either the reward distribution has a finite support and given each reward realization, the size distributions for different resources are perfectly positive correlated, or both the reward and the size of each request are resource-independent and the size distribution satisfies the same conditions. The same bound was previously known only for the single-resource case with unitary rewards \cite{arlotto2018logarithmic}. The authors left it as an open question whether the bound still holds when the reward could take multiple values. We provide an affirmative answer to this question, and even more, we generalize it to the multi-resource case. Notice that the numerical experiments in \cite{arlotto2018logarithmic} indicate that the $O(\log T)$ should be order tight even for the single-resource unitary-reward case, thus it is unlikely to improve our bound.

Our regret bounds are based on deriving a linear programming to serve as the upper bound of the expected reward collected by the prophet. A key step is to establish a closed-form characterization of the optimal LP solution. We observe that the optimal LP solution possesses a threshold structure, which motivates us to consider an adaptive threshold policy. Also, the closed-form characterization allows us to derive certain convexity that is crucial to prove our regret bounds. As discussed in Section 3, our LP-based approach is different from that in \cite{arlotto2018logarithmic}. Moreover, our approach also allows us to analyze how the expected reward obtained by the prophet scales over $T$, which in turn establishes the asymptotic optimality of our policies when $T\rightarrow\infty$.

Note that several recent papers assume that both the reward distribution and the size distribution have a finite support and derive uniform regret bounds that are independent of $T$ (e.g. \cite{jasin2012re}, \cite{bumpensanti2020re}, \cite{arlotto2019uniformly}, \cite{vera2020bayesian}, \cite{vera2019online}, \cite{freund2019uniform}). However, these bounds depend implicitly on the cardinality of the support of the reward and size. Moreover, it is shown in \cite{arlotto2019uniformly} that even for the simplest multi-secretary setting where there is only one resource and the size of each request equals $1$, the optimal regret will depend linearly on the reciprocal of the minimal mass of the reward and size distribution. Therefore, their methods and results could not be generalized to our setting where the size is assumed to be continuously distributed.

\subsection{Other Related Literature}
Our model is closely related to the knapsack problem and several streams of literature on online optimization. In addition to the results mentioned in the previous subsection, we briefly review other related results.

The knapsack problem is mainly studied under two different settings, the offline setting where the requests are all available to be served before any decision is made and the online setting where the requests arrive sequentially over time. The early formulation of the offline setting, as described in \cite{dantzig1957discrete}, features deterministic reward and size for each request. Then, an offline and stochastic setting is considered in the literature, where the reward and size of each request are assumed to be random and will be revealed only after being included in the knapsack (e.g. \cite{derman1975stochastic}, \cite{bhalgat2011improved}, \cite{ma2018improvements}). Notably, \cite{dean2008approximating} considered the setting where the reward is deterministic and the size is random for each request, and provided constant bound between optimal nonadaptive policy and optimal adaptive policy, which is the first result that addressed the \textit{benefit of adaptivity}. For the online setting, when requests arrive in an adversary order, \cite{marchetti1995stochastic} show that no online algorithm can achieve a constant competitive ratio. However, \cite{zhou2008budget} prove a parametric competitive ratio of $\log(U/L)+1$ where $L$ and $U$ are the lower and upper bounds, respectively, of the reward/size ratios of the requests. Moreover, \cite{babaioff2007knapsack} obtain constant competitive ratio when the requests arrive in a random order and \cite{dutting2017prophet} prove a $0.2$ competitive ratio when the requests arrive in a non-stationary stochastic order. Besides the constant bounds, \cite{marchetti1995stochastic} prove a $O(\log^{3/2} T)$ regret bound when both the reward and the size of each requests are independently and uniformly distributed. \cite{lueker1998average} further improve the previous bound to $O(\log T)$ and show that the order of the bound is tight.

Our problem is also related to the Online Linear Programming (OLP) literature. The OLP problem takes a standard linear programming as the underlying form, while the columns of the constraint matrix and the corresponding coefficients of the objective function will arrive sequentially over time. The trade-off between algorithm competitiveness and resource capacity, denoted as $c$, is studied in the literature (e.g. \cite{buchbinder2009online}, \cite{agrawal2014dynamic}, \cite{devanur2019near}) and it is shown that in order to have a $1-\varepsilon$ competitive ratio, it is essential for $c$ to be at least $\Omega(1/\varepsilon^2)$, which implies that it is necessary for $c$ to be scaled up simultaneously with $T$ to have any regret bound sublinear in the prophet. Also, a recent work \cite{li2019online} adopted an asymptotic setting that $c$ scales up linearly in $T$ and proved a $O(\log T\log\log T)$ regret bound. However, a term $\frac{T}{c}$ shows up in their regret bound and is abbreviated in notation $O(\cdot)$ since it is a constant in their framework. As a result, a direct application of their result to our problem would imply a regret at least linear in $T$ since we consider our problem under the setting that $c$ is fixed as $T\rightarrow\infty$. The above discussions imply that, although the formulation of OLP seems like a non-parametric generalization of our model, the fixed resource capacity setting considered in this paper essentially differentiate our problem from the OLP problem and would require analysis other than those in the OLP literature. Thus, though closely related, our work is not covered by the OLP literature and serves as a supplement to online decision making literature under an asymptotic setting other than that of the OLP literature.
\section{Problem Formulation}
In our problem, there are $m$ resources. Each resource $j$ has an initial capacity $c_j$. The resources are used to serve a stream of requests arriving sequentially over $T$ discrete time periods. At each time period $t$, one request arrives, denoted as request $t$, which is associated with a nonnegative reward $\mathbf{\tilde{r}^t}=[\tilde{r}_1^t,\tilde{r}_2^t,\dots,\tilde{r}_m^t]$ and a positive size $\mathbf{\tilde{d}^t}=[\tilde{d}_1^t,\tilde{d}_2^t,\dots,\tilde{d}_m^t]$. Here, $(\mathbf{\tilde{r}^t}, \mathbf{\tilde{d}^t})$ is assumed to be stochastic and its value is drawn independently from a known distribution which is stationary across time. After the value of $(\mathbf{\tilde{r}^t}, \mathbf{\tilde{d}^t})$ is revealed, denoted as $(\mathbf{r}^t,\mathbf{d}^t)$, the decision maker has to irrevocably decide to reject or accept request $t$. If rejected, no resource will be consumed and no reward will be collected. If accepted, the decision maker has to assign one of the resources to serve it. If resource $j$ is assigned to serve request $t$, then $d_j^t$ units of resource $j$ will be consumed and a reward $r_j^t$ will be collected. The objective of the decision maker is to maximize the total collected reward without violating the capacity constraints.

Any policy $\pi$ is defined via a sequence of binary variables $\{x^\pi_{j}(t)\}$, where $x^\pi_j(t)$ denotes whether request $t$ is served by resource $j$. Here, $x^\pi_j(t)$ is random and its realization is dependent on the arriving sequence of the requests. Policy $\pi$ is feasible as long as $\pi$ is non-anticipating, i.e., for any $j$, $x^\pi_{j}(t)$ can only depend on $\left\{(\mathbf{r}^1,\mathbf{d}^1),(\mathbf{r}^2,\mathbf{d}^2),\dots,(\mathbf{r}^t,\mathbf{d}^t)\right\}$, and $\pi$ satisfies the following constraints:
\begin{equation}\label{df1}
\sum_{t=1}^{T} d_j^t\cdot x^\pi_{j}(t)\leq c_j~~~\forall j\text{~~and~~}\sum_{j=1}^{m} x^\pi_{j}(t)\leq1~~~\forall t
\end{equation}
The total reward collected by policy $\pi$ is denoted as:
\begin{equation}\label{df2}
V^\pi(\mathbf{I})=\sum_{t=1}^{T}\sum_{j=1}^{m}r_{j}^t\cdot x^\pi_{j}(t)
\end{equation}
where $\mathbf{I}=\left\{(\mathbf{r}^1,\mathbf{d}^1),(\mathbf{r}^2,\mathbf{d}^2),\dots,(\mathbf{r}^T,\mathbf{d}^T)\right\}$ denotes the sample path of request arrivals over the entire time horizon. Then the expected total reward collected by the policy $\pi$ is denoted as $\mathbb{E}[V^\pi(\mathbf{I})]$, where $\mathbb{E}[\cdot]$ denotes taking expectation over the request arrivals $\mathbf{I}$.

We compare the performance of any feasible policy $\pi$ to the prophet, an offline decision maker with full knowledge of the arriving sequence of the requests and always applying the optimal policy in hindsight. Similarly, given the sample path $\mathbf{I}$, we use a binary variable $x^t_{j}(\mathbf{I})$ to denote whether request $t$ is served by resource $j$ under the optimal policy in hindsight. Here, $x^t_{j}(\mathbf{I})$ is allowed to depend on the whole sample path $\mathbf{I}$. Indeed, given $\mathbf{I}$, $\{x^t_{j}(\mathbf{I})\}$ is the optimal solution to the following optimal offline problem:
\begin{equation}\label{df3}
\begin{aligned}
V^{\text{off}}(\mathbf{I})=&\max~~\sum_{t=1}^{T}\sum_{j=1}^{m}r_{j}^t\cdot x_{j}^t\\
&~\mbox{s.t.}~~~\sum_{t=1}^{T} d_j^t\cdot x_{j}^t\leq c_j~~~\forall j\\
&~~~~~~~~\sum_{j=1}^{m} x_{j}^t\leq1~~~\forall t\\
&~~~~~~~~x_{j}^t\in\{0,1\}~~~~ \forall j, \forall t
\end{aligned}
\end{equation}
Then, the expected total reward collected by the prophet is denoted as $\mathbb{E}[V^{\text{off}}(\mathbf{I})]$.

We will use \textit{regret} to measure the performance of any feasible policy. The regret of the policy $\pi$, denoted as $\text{Regret}(\pi)$, is defined as the difference between the expected total reward collected by the prophet and the expected total reward collected by the policy $\pi$:
\begin{equation}\label{df4}
\text{Regret}(\pi)=\mathbb{E}[V^{\text{off}}(\mathbf{I})]-\mathbb{E}[V^\pi(\mathbf{I})]
\end{equation}

\section{Regret Bound for Finite-Support Reward Case}
In this section, we develop our policy when the reward distribution has a finite support. Specifically, we make the following assumption:
\begin{assumption}\label{newassump1}
For each $t$, the value of $~\mathbf{\tilde{r}^t}$ is drawn independently from a finite set $\{ \mathbf{r}_1,\mathbf{r}_2,\dots,\mathbf{r}_n \}$, where $\mathbf{r}_i=[r_{i1}, r_{i2}, \dots, r_{im}]$. Moreover, given that $\mathbf{\tilde{r}^t}$ is realized as $\mathbf{r}_i$, we have $\mathbf{\tilde{d}^t}=\mathbf{b}_i\cdot \tilde{u}_i^t$, where $\mathbf{b}_i=[b_{i1}, b_{i2}, \dots, b_{im}]$ denotes the positive deterministic weight and $\tilde{u}_i^t$ is a single-dimensional random variable.
\end{assumption}
For each $t$, we use $p_i$ to denote the probability that $\mathbf{\tilde{r}^t}$ is realized as $\mathbf{r}_i$ and use $F_i(\cdot)$ to denote the distribution of $\tilde{u}_i^t$, where its realization is denoted as $u_i^t$. Since the hindsight optimum obtained by solving the integer programming \eqref{df3} often preserves complex structure thus is very hard to analyze, we first derive a tractable upper bound of $\mathbb{E}[V^{\text{off}}(\mathbf{I})]$. Moreover, we show that under Assumption \ref{newassump1}, the prophet upper bound we have derived could be reduced equivalently into $n$ sub-LPs and accordingly, our original online resource allocation problem could be reduced to $n$ separate sub-problems, where in each sub-problem the reward of each request is deterministic and each sub-LP serves as the prophet upper bound for each sub-problem.

Then, we show that the optimal solution of each sub-LP possesses a threshold structure and the thresholds could be given in a closed-form expression. By assuming some additional regularity conditions on the distribution function $F_i(\cdot)$ for each $i$, we further show a convex structure of the thresholds and then obtain a second-order upper estimates of how the optimal value of each sub-LP would vary over the right hand side (RHS) of the constraints. By applying a adaptive threshold policy in each sub-problem, we show in our regret analysis that the first order terms in the estimate will cancel out and only the second order terms will be left, which will finally lead to a $O(\log T)$ regret bound.
\subsection{Prophet Upper Bound}
We obtain an upper bound of $\mathbb{E}[V^{\text{off}}(\mathbf{I})]$ by considering the following LP:
\begin{equation}\label{newna1}
\begin{aligned}
\text{LP}^{\text{UB}}:=&\max~~T\cdot \sum_{i=1}^{n} p_i\cdot \mathbb{E}_i[\sum_{j=1}^{m}r_{ij}\cdot x_{ij}(u_i)]\\
&~\mbox{s.t.}~~~T\cdot \sum_{i=1}^{n} p_i\cdot b_{ij}\cdot \mathbb{E}_i[u_i\cdot x_{ij}(u_i)]\leq c_{j}~~~\forall j\\
&~~~~~~~~\sum_{j=1}^{m}x_{ij}(u_i)\leq 1~~~\forall i, \forall u_i\\
&~~~~~~~~0\leq x_{ij}(u_i)\leq 1~~~\forall i, \forall j, \forall u_i,
\end{aligned}
\end{equation}
Here, $\mathbb{E}_i[\cdot]$ denotes taking expectation over realization $u_i$ with distribution function $F_i(\cdot)$ and $x_{ij}(u_i)$ could be interpreted as the probability that request $t$ is served using resource $j$ given the realization $(\mathbf{r}_i,\mathbf{b}_i\cdot u_i)$, regardless of the time period $t$.
\begin{lemma}\label{newlemma10}
It holds that $\text{LP}^{\text{UB}}\geq\mathbb{E}[V^{\text{off}}(\mathbf{I})]$.
\end{lemma}
\begin{proof}{Proof:}
We use $\{x^t_j(\mathbf{I})\}$ to denote one optimal solution of \eqref{df3} given the arriving sequence of requests $\mathbf{I}$ and denote $x^t_{ij}(u_i)=\mathbb{E}[x^t_j(\mathbf{I})|(\mathbf{r}^t,\mathbf{d}^t)=(\mathbf{r}_i,\mathbf{b}_i\cdot u_i)]$. Then we have
\begin{equation}\label{newlpupper1}
0\leq x^t_{ij}(u_i)\leq 1\text{~~and~~}\sum_{j=1}^{m}x^t_j(\mathbf{I})\leq 1\Rightarrow \sum_{j=1}^{m}x^t_{ij}(u_i)\leq 1
\end{equation}
Also, we have that
\begin{equation}\label{newlpupper2}
\begin{aligned}
\mathbb{E}[V^{\text{off}}(\mathbf{I})]&=\mathbb{E}[\sum_{t=1}^{T}\sum_{j=1}^{m}r_j^t\cdot x^t_j(\mathbf{I})]=\sum_{t=1}^{T}\mathbb{E}\left[\mathbb{E}[\sum_{j=1}^{m}r_j^t\cdot x^t_j(\mathbf{I})|(\mathbf{r}^t,\mathbf{d}^t)=(\mathbf{r}_i,\mathbf{b}_i\cdot u_i)]\right]\\
&=\sum_{t=1}^{T}\mathbb{E}[\sum_{j=1}^{m}r_j\cdot x^t_{ij}(u_i)]
\end{aligned}
\end{equation}
For each $j$, we could further obtain the following constraint regarding $x^t_{ij}(u_i)$:
\begin{equation}\label{newlpupper3}
\begin{aligned}
\sum_{t=1}^{T}\sum_{i=1}^{n}p_i\cdot\mathbb{E}_i[b_{ij}\cdot u_i\cdot x^t_{ij}(u_i)]&=\sum_{t=1}^{T}\mathbb{E}\left[\mathbb{E}[d_j^t\cdot x^t_j(\mathbf{I})|(\mathbf{r}^t,\mathbf{d}^t)=(\mathbf{r}_i,\mathbf{b}_i\cdot u_i)]\right]\\
&=\mathbb{E}[\sum_{t=1}^{T}d_j^t\cdot x^t_j(\mathbf{I})]\leq c_j
\end{aligned}
\end{equation}
Since the request distribution is stationary across different time periods and the offline optimal solution made by the prophet is not influenced by the order of the sequence $\mathbf{I}$, there must exists one optimal solution of \eqref{df3} such that $x^{t_1}_{ij}(u_i)=x^{t_2}_{ij}(u_i)$ for arbitrary $t_1$ and $t_2$. Thus, we could suppress the notation $t$ in $x^t_{ij}(u_i)$. Then given \eqref{newlpupper1}, \eqref{newlpupper2} and \eqref{newlpupper3}, it holds that $\text{LP}^{\text{UB}}\geq \mathbb{E}[V^{\text{off}}(\mathbf{I})]$.
\Halmos
\end{proof}

We then show how to reduce \eqref{newna1} equivalently into $n$ sub-LPs and characterize the optimal solution of each sub-LP in a closed-form expression. We consider the sub-LP of the following formulation:
\begin{equation}\label{up22}
\begin{aligned}
G_i(\mathbf{\mathbf{\tilde{c}}}_i,t):=&\max~~(T-t+1)\cdot p_i\cdot \mathbb{E}_i[\sum_{j=1}^{m}r_{ij}\cdot x_{ij}(u_i)]\\
&~\mbox{s.t.}~~~(T-t+1)\cdot p_i\cdot b_{ij}\cdot \mathbb{E}_i[u_i\cdot x_{ij}(u_i)]\leq \tilde{c}_{ij}~~~\forall j\\
&~~~~~~~~\sum_{j=1}^{m}x_{ij}(u_i)\leq 1~~~\forall u_i\\
&~~~~~~~~0\leq x_{ij}(u_i)\leq 1~~~\forall j, \forall u_i
\end{aligned}
\end{equation}
where $\mathbf{\tilde{c}}_i=[\tilde{c}_{i1},\tilde{c}_{i2},\dots,\tilde{c}_{im}]$.
\begin{lemma}\label{lemma1101}
It holds that
\begin{equation}\label{up21}
\begin{aligned}
\text{LP}^{\text{UB}}=&\max~~\sum_{i=1}^{n}G_i(\mathbf{\tilde{c}}_i, 1)\\
&~\mbox{s.t.}~~~\sum_{i=1}^{n}\tilde{c}_{ij}\leq c_j~~~ \forall j\\
&~~~~~~~~\tilde{c}_{ij}\geq0~~~ \forall i, \forall j
\end{aligned}
\end{equation}
\end{lemma}
~\\
We now show the threshold structure of the optimal solution of the sub-LP \eqref{up22}. Specifically, for each $i$, denote $\{i_1,i_2,\dots,i_m\}$ as a permutation of $\{1,2,\dots,m\}$ such that $r_{ii_1}\geq r_{ii_2}\geq\dots\geq r_{ii_m}\geq0$ and we set $r_{ii_{m+1}}=0$, then for each $i, j, t$,  we define the threshold $\mu^{ii_j}_t(\mathbf{\tilde{c}}_i)$ as the solution of the following equation:
\begin{equation}\label{up2}
p_i\cdot\int_{0}^{\mu^{ii_j}_t(\mathbf{\tilde{c}}_i)}u dF_i(u)=\frac{\sum_{k=1}^{j}\tilde{c}_{ii_k}/b_{ii_k}}{T-t+1}~~~~
\end{equation}
If there is no $\mu^{ii_j}_t(\mathbf{\tilde{c}}_i)$ satisfying the above equation, we set $\mu^{ii_j}_t(\mathbf{\tilde{c}}_i)=\infty$.
\begin{lemma}\label{Upper1}
For each $i, j, u_i$, define the solution $\hat{x}_{ii_j}(u_i)$ by
\begin{equation}\label{sep02}
\hat{x}_{ii_j}(u_i)=\left\{\begin{aligned}
&1,~~~\text{if~}\mu^{ii_{j-1}}_t(\mathbf{\tilde{c}}_i)< u_i\leq \mu^{ii_{j}}_t(\mathbf{\tilde{c}}_i)\\
&0,~~~\text{otherwise}
\end{aligned}\right.
\end{equation}
where $\mu^{ii_{0}}_t(\mathbf{\tilde{c}}_i)$ is set to be $0$. Then $\{\hat{x}_{ii_j}(u_i)\}$ is an optimal solution to \eqref{up22} and it holds that
\begin{equation}\label{sep03}
G_i(\mathbf{\tilde{c}}_i,t)=(T-t+1)\cdot p_i\cdot \sum_{j=1}^{m}(r_{ii_j}-r_{ii_{j+1}})\cdot F_i(\mu^{ii_j}_t(\mathbf{\tilde{c}}_i))
\end{equation}
\end{lemma}
The proof of Lemma \ref{lemma1101} and Lemma \ref{Upper1} is relegated to the appendix. Together, Lemma \ref{lemma1101} and Lemma \ref{Upper1} will provide a closed-form characterization of the optimal solution of the linear programming prophet upper bound \eqref{newna1}. Indeed, notice that $G_i(\mathbf{\tilde{c}}_i,t)$ is a concave function over $\mathbf{\tilde{c}}_i$, the optimization problem \eqref{up21} is a convex optimization problem, and we could use the following formula to obtain the supergradient of $G_i(\mathbf{\tilde{c}}_i,t)$ immediately:
\[
\frac{\partial G_i(\mathbf{\tilde{c}}_i,t)}{\partial \tilde{c}_{i_j}}=\sum_{k=j}^{m}\frac{r_{ii_k}-r_{ii_{k+1}}}{b_{ii_j}\cdot\mu^{ii_k}_t(\mathbf{\tilde{c}}_i)}, ~~~\forall j
\]
Thus, the optimization problem \eqref{up21} is tractable and could be solved efficiently.\\
\textbf{Remark:} when there is only one resource and the reward is unitary, i.e., $m=n=1$, our prophet upper bound \eqref{newna1} is identical to the upper bound in \cite{arlotto2018logarithmic}. However, their approach to derive the prophet upper bound is based on an observation that when there is only one resource and the reward is unitary, the hindsight optimal policy is to simply serve the requests in an increasing order of the size. Thus, the order statistics of the size of the requests is enough to characterize the hindsight optimum. Their approach could not be directly applied to a more general setting where this observation doesn't hold. Specifically, when there are multiple resources and the reward could take multiple values, the hindsight optimum is an integer programming \eqref{df3}, the solution of which should not be serving the request from the smallest size. Instead, our approach is based on directly dealing with the linear programming prophet upper bound \eqref{newna1} and show its specific threshold structure.
\subsection{Regularity Conditions}
In this subsection, we will describe the regularity conditions that the distribution functions $F_i(\cdot)$ need to satisfy. A special case of our problem where there is only one resource and the reward is unitary is identical to the dynamic and stochastic knapsack problem with equal reward studied in a stream of literature: \cite{coffman1987optimal}, \cite{bruss1991wald}, \cite{papastavrou1996dynamic}, \cite{arlotto2018logarithmic}. All these papers assumed some regularity conditions on the distribution functions to obtain their results. One may see the necessity of the regularity conditions by referring to Section 5 in \cite{papastavrou1996dynamic} which shows that one could not expect some structural properties such as the monotonicity of the optimal threshold and the concavity of the optimal value for general distributions. Specifically, \cite{arlotto2018logarithmic} proved a $O(\log T)$ regret bound for this knapsack problem by assuming the distribution function belongs to the so-called $\textit{typical~class}$. In what follows, we adopt the $\textit{typical~class}$ conditions with only one minor change which will be discussed later.
\begin{assumption}\label{assumption1}
For some $\bar{\omega}>0$ and for all $i$, the distribution function $F_i(\cdot)$ has a continuous density function $f_i(\cdot)$ and satisfies the following two conditions:
\begin{enumerate}
  \item (Behavior at Zero) There exists a constant $0<\lambda<1$ and two constants $1<\gamma_1<\gamma_2$ such that
  \begin{equation}\label{assump1}
  \gamma_1\cdot F_i(\lambda w)\leq F_i(w)\leq\gamma_2\cdot F_i(\lambda w)~~~\text{for~all~}w\in(0,\bar{\omega})
  \end{equation}
  \item (Monotonicity) The function $w:\rightarrow w^3f_i(w)$ is a non-decreasing function over $(0,\bar{\omega})$.
\end{enumerate}
\end{assumption}
The only difference between our Assumption \ref{assumption1} and the $\textit{typical~class}$ conditions in \cite{arlotto2018logarithmic} lies in condition $1$. $\textit{Typical~class}$ only assumed $F_i(w)$ to be lower bounded by $F_i(\lambda w)$, i.e., they only assumed $\gamma_1\cdot F_i(\lambda w)\leq F_i(w)$, while in our assumption, we assume $F_i(w)$ to be both lower bounded and upper bounded by $F_i(\lambda w)$. The discussions in Section 5 of \cite{arlotto2018logarithmic} showed that their $\textit{typical~class}$ conditions could cover a very broad class of continuous distributions. Moreover, it is easy to check that the following distributions they have considered continue to satisfy both conditions in Assumption \ref{assumption1}:
\begin{enumerate}
  \item Uniform distribution, exponential distribution and logit-normal distribution.
  \item Truncated normal distribution and truncated logistic distribution.
  \item Power distribution with distribution function $F(x)=Ax^\alpha$ for $A, \alpha>0$ truncated on $(0,b)$, where $b$ is any positive constant.
  \item Distributions with convex distribution function $F(\cdot)$ and continuous density function $f(\cdot)$ such that $f(0)>0$.
  \item The mixture distribution of some distributions satisfying both conditions in Assumption \ref{assumption1}.
\end{enumerate}
We show in the next lemma that condition $1$ in Assumption \ref{assumption1} will enable us to derive the following estimate over the distribution functions $F_i(\cdot)$.
\begin{lemma}\label{assumplemma1}
If there exists a continuous distribution $F(\cdot)$ such that for some $\bar{\omega}>0$, condition $1$ in Assumption \ref{assumption1} is satisfied, i.e., there exists a constant $0<\lambda<1$ and two constants $1<\gamma_1<\gamma_2$ such that
  \begin{equation}
  \gamma_1\cdot F(\lambda w)\leq F(w)\leq\gamma_2\cdot F(\lambda w)~~~\text{for~all~}w\in(0,\bar{\omega})
  \end{equation}
then there exists a constant $M>1$ such that for any $w_1, w_2$ satisfying $0<w_1<w_2<\bar{\omega}$, we have
\begin{equation}
\frac{w_2(F(w_2)-F(w_1))}{\int_{w_1}^{w_2}wdF(w)}\leq M
\end{equation}
\end{lemma}
Also, we could obtain the following convex property regarding the threshold $\mu^{ii_j}_t(\mathbf{\tilde{c}}_i)$ for each $i$ and $j$ under condition $2$ in Assumption \ref{assumption1}.
\begin{lemma}\label{newlemma1}
Suppose Assumption $\ref{assumption1}$ holds, then for each $i, j$, and any $k=1,2,\dots,j$, the function $\frac{1}{\mu^{ii_j}_t(\mathbf{\tilde{c}}_i)}$ is a convex function over $\tilde{c}_{ii_k}$.
\end{lemma}
The proofs of the above two lemmas are relegated to the appendix. Together, Lemma \ref{assumplemma1} and Lemma \ref{newlemma1} would enable us to obtain the second-order upper estimate of the sensitivity of $G_i(\mathbf{\tilde{c}}_i,t)$ over $\mathbf{\tilde{c}}_i$. Indeed, given the formulation of $G_i(\mathbf{\tilde{c}}_i,t)$ in \eqref{sep03}, it is enough to consider the sensitivity of $F_i(\mu^{ii_j}_t(\mathbf{\tilde{c}}_i))$ over $\mathbf{\tilde{c}}_i$ for each $j$. The following second-order estimate generalizes the estimate in \cite{arlotto2018logarithmic} from the single-resource setting to the multi-resource setting and the proof could be found in the appendix.
\begin{lemma}\label{newlemma2}
Suppose Assumption \ref{assumption1} holds, then for each $i, j$ and any $k=1,2,\dots,j$, we have that
\begin{equation}\label{sep04}
\begin{aligned}
F_i(\mu^{ii_j}_{t}(\mathbf{\tilde{c}}_i))-F_i(\mu^{ii_j}_{t}(\mathbf{\tilde{c}}_i-b_{ii_k}u\cdot\mathbf{e}_{i_k}))
\leq\frac{(M-1)\cdot u^2}{(\tilde{c}_{ii_k}/b_{ii_k}) p_i(T-t+1)\mu^{ii_j}_{t}(\mathbf{\tilde{c}}_i)}+\frac{u}{p_i(T-t+1)\mu^{ii_j}_{t}(\mathbf{\tilde{c}}_i)}
\end{aligned}
\end{equation}
holds for any $u\in[0,\tilde{c}_{ii_k}/b_{ii_k}]$.
\end{lemma}
\subsection{Adaptive Threshold Policy}
We now derive our adaptive threshold policy and prove the corresponding regret bound. Since the optimal solution and the sensitivity of each sub-LP \eqref{up22} could be fully characterized in \eqref{sep02} and in \eqref{sep04}, we will reduce the original online resource allocation problem into $n$ sub-problems to utilize this characterization. Specifically, we will divide each resource $j$ into $n$ parts and assign a capacity $c^*_{ij}$ to each $i$th part of resource $j$, where $\{c^*_{ij}\}$ is denoted as the optimal solution of \eqref{up21}. Then, for each request with a realized reward $\mathbf{r}_i$, we will only consider using the $i$th part of each resource to serve it. By doing so, the $i$th sub-problem is only concerned about serving request with a deterministic reward $\mathbf{r}_i$ using a separate part of each resource. Thus, each sub-problem will be independent from each other and it is enough to focus on each sub-problem separately.

For $i$th sub-problem, we will construct our control policy based on re-solving \eqref{up2} to obtain the thresholds adaptively. Specifically, at each time period $t$, denote the remaining capacity of the $i$th part of resource $j$ as $c^t_{ij}$ and denote $\mathbf{c}^t_i=[c^t_{i1}, c^t_{i2}, \dots, c^t_{im}]$, then given $\mathbf{\tilde{r}}^t=\mathbf{r}_i$ and $\mathbf{\tilde{d}}^t=\mathbf{b}_i\cdot u_i^t$, the threshold structure of the optimal solution of $G_i(\mathbf{c}^t_i, t)$ in \eqref{sep02} implies us to assign resource $i_j$ to serve the request as long as $\mu^{ii_{j-1}}_t(\mathbf{c}^t_i)<u^t_i\leq\mu^{ii_j}_t(\mathbf{c}^t_i)$. However, the actual threshold will be truncated by the remaining capacity $\mathbf{c}^t_i$. Thus, in our adaptive threshold policy, we will compare $u^t_i$ sequentially to the truncated thresholds $h^{ii_j}_t(\mathbf{c}^t_i)=\min\{\mu^{ii_j}_t(\mathbf{c}^t_i), c^t_{ii_j}/b_{ii_j}\}$ from $j=1$ to $m$ and assign resource $i_j$ to serve the request once $u^t_i\leq h^{ii_j}_t(\mathbf{c}^t_i)$. Our adaptive threshold policy is formally presented in Algorithm \ref{ATP1}.\\
\begin{algorithm}
\begin{algorithmic}[1]
\caption{Adaptive Threshold Policy ($\text{ATP}_1$)}
\label{ATP1}
\STATE Solve \eqref{up21} and denote the optimal solution as $\{\mathbf{c}^1_i\}$.
\STATE At time period $t=1,2,\dots,T$:
\STATE After request $t$ arrives and reveals itself as $(\mathbf{r}_i, \mathbf{b}_i\cdot u^t_i)$, re-solve \eqref{up2} to obtain the thresholds $\mu^{ii_j}_t(\mathbf{c}^t_i)$ for each $j$.
\STATE For each $j$, define the truncated threshold $h^{ii_j}_t(\mathbf{c}^t_i)=\min\{ c^t_{ii_j}/b_{ii_j},\mu^{ii_j}_t(\mathbf{c}^t_i) \}$ and then do the following:
\begin{enumerate}[label=(\roman*)]
  \item If for all $j$, $u^t_i>h^{ii_j}_t(\mathbf{c}^t_i)$, reject request $t$ and then set $c^{t+1}_{i'j'}=c^t_{i'j'}$ for each $i'$ and $j'$.
  \item Otherwise, compare $u^t_i$ to the truncated thresholds $h^{ii_1}_t(\mathbf{c}^t_i), h^{ii_2}_t(\mathbf{c}^t_i),\dots,h^{ii_m}_t(\mathbf{c}^t_i)$ sequentially and assign the first resource $i_j$ to serve request $t$ such that $u^t_i\leq h^{ii_j}_t(\mathbf{c}^t_i)$. Then set $c^{t+1}_{ii_j}=c^t_{ii_j}-b_{ii_j}\cdot u^t_i$ and set $c^{t+1}_{i'j'}=c^t_{i'j'}$ for all other $(i', j')$ such that $(i', j')\neq(i,i_j)$.
\end{enumerate}
\end{algorithmic}
\end{algorithm}
Next, we discuss the performance of our policy. For each $i$, denote $V^{\text{ATP}_1}_i(\mathbf{I})$ as the total reward collected by Algorithm \ref{ATP1} for the $i$th sub-problem, i.e., the total reward collected from serving the requests with a realized reward $\mathbf{r}_i$, then we have
\begin{equation}\label{PR4}
\mathbb{E}[V^{\text{ATP}_1}(\mathbf{I})]=\sum_{i=1}^{n}\mathbb{E}[V^{\text{ATP}_1}_i(\mathbf{I})]
\end{equation}
Thus, from \eqref{up21}, the regret of Algorithm \ref{ATP1} could be expressed as the summation of the regret of each sub-problem:
\begin{equation}\label{v2PR1}
\text{Regret}(\text{ATP}_1)\leq\sum_{i=1}^{n}G_i(\mathbf{c}^*_i,1)-\mathbb{E}[V^{\text{ATP}_1}_i(\mathbf{I})]
\end{equation}
where $\{\mathbf{c}^*_i\}$ denotes the optimal solution of \eqref{up21}. For each $i$, we have the following upper bound of the term $G_i(\mathbf{c}^*_i,1)-\mathbb{E}[V^{\text{ATP}_1}_i(\mathbf{I})]$.
\begin{lemma}\label{Main1}
Suppose Assumption \ref{assumption1} holds and let $K_i$ be a constant defined as follows:
\begin{equation}\label{PR2}
K_i=\left\lceil \frac{\sum_{j=1}^{m}c^*_{ii_j}/b_{ii_j}}{p_i\cdot\int_{0}^{\bar{\omega}}wdF_i(w)} \right\rceil
\end{equation}
Then, we have the following bound:
\begin{equation}\label{maineq1}
\begin{aligned}
G_i(\mathbf{c}_i^*,1)-\mathbb{E}[V^{\text{ATP}_1}_i(\mathbf{I})]&\leq \sum_{j=1}^{m}r_{ii_j}[(j+1)M-j](\log T+1)+r_{\max}^i\cdot K_i=O(m^2\log T)
\end{aligned}
\end{equation}
where $r_{\max}^i=\max\{r_{i1},r_{i2},\dots,r_{im}\}$ and $M$ is the constant defined in Lemma \ref{assumplemma1}.
\end{lemma}
The proof of Lemma \ref{Main1} is based on first deriving the recursion of the "to-go" expected reward of the policy given the current time period $t$ and the remaining capacity $\mathbf{c}_i^t$, denoted as $V_t(\mathbf{c}_i^t)$, in the following way:
\[
V_t(\mathbf{c}^t_i)=(1-p_i\cdot F(h_t^{ij_W}(\mathbf{c}^t_i)))\cdot V_{t+1}(\mathbf{c}_i^t)+p_i\cdot\sum_{w=1}^{W}\int_{h_t^{ij_{w-1}}(\mathbf{c}_i^t)}^{h_t^{ij_w}(\mathbf{c}_i^t)}\{ \tr_{j_w}+V_{t+1}(\mathbf{c}_i^t-b_{ij_w}u\cdot\mathbf{e}_{j_w})\}dF_i(u)
\]
where $\{h_t^{ij_w}(\mathbf{c}_i^t)\}$ denotes the monotone increasing subsequence of $\{h_t^{ii_j}(\mathbf{c}_i^t)\}$. Note that although the thresholds $\{\mu_t^{ii_j}(\mathbf{c}_i^t)\}$ is monotone increasing in $j$, the actual thresholds $\{h_t^{ii_j}(\mathbf{c}_i^t)\}$ involved in our policy may not be monotone due to the truncation of the remaining capacity, thus, only the monotone subsequence $\{h_t^{ij_w}(\mathbf{c}_i^t)\}$ will show up in the above recursion. Then denoting $\rho_t(\mathbf{c}_i^t)=G_i(\mathbf{c}_i^t,t)-V_t(\mathbf{c}^t_i)$, we could substitute $\rho_t(\mathbf{c}_i^t)$ into the above formula and obtain a recursion over $\rho_t(\mathbf{c}_i^t)$. We will use the second-order upper estimate in Lemma \ref{newlemma2} to upper bound each term in the recursion, and together with an analysis on the influence of truncation, i.e., comparing the subsequence $\{h_t^{ij_w}(\mathbf{c}_i^t)\}$ with $\{\mu_t^{ii_j}(\mathbf{c}_i^t)\}$, we could show that the first-order term will cancel out in the recursion and only the second-order term will be left. Thus, we show that for any $\mathbf{c}_i^t$, we have $\rho_t(\mathbf{c}_i^t)= O(\sum_{t'=T-t+1}^{T} \frac{1}{t'})$, which implies that $\rho_1(\mathbf{c}^*_i)=O(\log T)$. The complete proof is relegated to the appendix. Then, from \eqref{v2PR1}, we could establish the following regret bound of Algorithm \ref{ATP1}.
\begin{theorem}\label{Main3}
Suppose Assumption \ref{assumption1} hold, then we have the following regret bound:
\begin{equation}\label{PR9}
\begin{aligned}
\text{Regret}(\text{ATP}_1)\leq \sum_{i=1}^{n}\sum_{j=1}^{m}r_{ii_j}[(j+1)M-j](\log T+1)+\sum_{i=1}^{n}r^i_{\max}\cdot K_i=O(nm^2\log T)
\end{aligned}
\end{equation}
\end{theorem}
It is interesting to note that the re-solving procedure for each sub-problem in Algorithm \ref{ATP1} has also been applied to designing near-optimal policy for network revenue management problem with a so-called "non-degeneracy" assumption in \cite{jasin2012re} and following-up work \cite{jasin2015performance}. We will conclude this subsection by making a comparison between our approach and their approach in the following remark.\\
\textbf{Remark:} a similarity between our approach and the approach in \cite{jasin2012re} and \cite{jasin2015performance} is that the optimal solution for the LP upper bound could be expressed in closed-form via the remaining capacity of each resource. However, our regret analysis is fundamentally different from theirs. A key step in their analysis is that they show under the "non-degeneracy" assumption, the closed-form expression could be obtained provided that $\frac{c^t_j}{T-t+1}$ for all $j$, where $c^t_j$ denotes the remaining capacity of resource $j$ at time $t$, doesn't variate too much from $\frac{c_j}{T}$ so that the optimal bases for the LP upper bound doesn't change. Then they show that this condition could be guaranteed for a sufficiently long time in their policy by constructing a martingale involving $\frac{c^t_j}{T-t+1}$ and applying the Doob's inequality for martingale. As a result, their regret bound is obtained when the ratio $\frac{c_j}{T}$ is fixed and their approach could not be directly applied to the setting where $c_j$ could grow sublinearly in $T$. In contrast, our closed-form expression is obtained for any $c^t_j$ and our regret analysis is based on considering the curvature of the distribution function of the size instead of constructing a martingale, thus our approach works for the setting where $c_j$ is fixed and $T\rightarrow\infty$. We also need to analyze the influence of the fact that the actual threshold is truncated by the remaining capacity, which is not needed in their analysis since $c_j$ grows linearly in $T$ thus $c^t_j$ will always exceeds the size of the request except for the last few time periods. Moreover, the re-solving procedure in \cite{jasin2012re} and \cite{jasin2015performance} requires solving a LP, which may be time consuming, while our re-solving procedure in Algorithm \ref{ATP1} could be done quickly with the closed-form characterization \eqref{up2}.

\subsection{Asymptotic Optimality}
In this subsection we show how does $\mathbb{E}[V^{\text{off}}(\mathbf{I})]$ scale over $T$. Our next theorem shows that $\mathbb{E}[V^{\text{off}}(\mathbf{I})]=\Omega(T^{\frac{1}{1+\alpha}})$, where $\alpha$ is a positive constant depending on $\lambda, \gamma_2$ in Assumption \ref{assumption1}. Thus, we immediately show the asymptotic optimality of our adaptive threshold policy as $T\rightarrow+\infty$.
\begin{theorem}\label{Main2}
Suppose Assumption \ref{newassump1} and Assumption \ref{assumption1} hold, then we have
\begin{equation}\label{maineq2}
\mathbb{E}[V^{\text{off}}(\mathbf{I})]=\Omega(T^{\frac{1}{1+\alpha}})
\end{equation}
where $\alpha$ is a positive constant satisfying $\frac{1}{\gamma_2}=\lambda^\alpha$.
\end{theorem}
\section{Regret Bound for Resource-independent Case}
The main purpose of this section is to study whether the term $n$ in our previous $O(\log T)$ regret bound could be removed and whether we could directly deal with our online resource allocation problem without reducing the problem into $n$ separate problems. We show that under a special case where both the reward and the size of each request are resource-independent, we can indeed derive a similar adaptive threshold policy to achieve a $O(\log T)$ regret bound without any dependence on $n$. The main difference of our analysis from the previous analysis is that we directly show the threshold structure of the prophet upper bound without dividing it into $n$ sub-LPs. As a result, we could consider our online resource allocation problem as a whole without any reduction. Although we only analyze the case where the reward distribution has a finite support to keep the notation in consistent with the previous sections, our analysis could be directly applied to the case where the support of the reward distribution is infinite and our regret bound continues to hold. In addition to Assumption \ref{newassump1}, we further make the following assumption:
\begin{assumption}\label{newassump3}
For each $i$, $r_{ij}=r_i$ for every $j$ and $\mathbf{b}_i$ is an all-one vector .
\end{assumption}
In what follows, we will first show the threshold structure of our prophet upper bound and then derive our policy and prove the corresponding regret bound.
\begin{theorem}\label{TRmain1}
We have the following upper bound of the expected total reward collected by the prophet:
\begin{equation}\label{TR1}
\mathbb{E}[V^{\text{off}}(\mathbf{I})]\leq T\cdot\sum_{i=1}^{n}p_i\cdot r_i\cdot F_i(r_i\cdot\mu(\mathbf{c}))
\end{equation}
where the threshold $\mu(\mathbf{c})$ is obtained by solving the following equation:
\begin{equation}\label{TR2}
\sum_{i=1}^{n}p_i\cdot \int_{0}^{r_i\cdot\mu(\mathbf{c})}u_idF_i(u_i)=\frac{\sum_{j=1}^{m}c_j}{T}
\end{equation}
If there is no $\mu(\mathbf{c})$ satisfying the above equation, we set $\mu(\mathbf{c})=+\infty$.
\end{theorem}
\begin{proof}{Proof:}
First, we have that
\begin{equation}\label{PR12}
\begin{aligned}
\text{LP}^{\text{UB}}\leq &\max~~T\cdot\sum_{i=1}^{n}p_i\cdot \mathbb{E}_i[ r_i\cdot y_i(u_i)]\\
&~\mbox{s.t.}~~~T\cdot\sum_{i=1}^{n}p_i\cdot\mathbb{E}_i[d\cdot y_i(u_i)]\leq \sum_{j=1}^{m}c_j\\
&~~~~~~~~0\leq y_i(u_i)\leq1~~~ \forall i, \forall u_i\\
\end{aligned}
\end{equation}
The inequality holds by noting that the solution $\{\hat{y}_i(u_i)\}$ defined by $\hat{y}_i(u_i)=\sum_{j=1}^{m}x^*_{j}(r_i,u_i)$, where $\{x^*_{j}(r_i,u_i)\}$ is the optimal solution of \eqref{newna1}, is feasible to the RHS of \eqref{PR12}, and we have
\[
\text{LP}^{\text{UB}}=T\cdot\sum_{i=1}^{n}p_i\cdot \mathbb{E}_i[\sum_{j=1}^{m} r_i\cdot x^*_j(r_i,u_i)]=T\cdot\sum_{i=1}^{n}p_i\cdot \mathbb{E}_i[ r_i\cdot \hat{y}_i(u_i)]
\]
Suppose the optimal solution of the RHS of \eqref{PR12} is denoted as $\{y^*_i(u_i)\}$, then define $\hat{z}_i=\mathbb{E}_i[y^*_i(u_i)]\in[0,1]$ for every $i$ and we have the following constraint regarding $\hat{z}_i$:
\begin{equation}\label{PR14}
\begin{aligned}
\sum_{j=1}^{m}c_i&\geq T\cdot\sum_{i=1}^{n}p_i\cdot \mathbb{E}_i[u_i\cdot y_i^*(u_i)]\geq T\cdot\sum_{i=1}^{n}p_i\cdot\left[\int_{0}^{F_i^{-1}(\hat{z}_i)}u_idF_i(u_i)\right]=T\cdot\sum_{i=1}^{n}p_i\cdot\left[\int_{0}^{\hat{z}_i}F_i^{-1}(v_i)dv_i\right]
\end{aligned}
\end{equation}
where the second inequality holds by noting that $y_i^*(u_i)\in[0,1]$ thus $\mathbb{E}_i[u_i\cdot y_i^*(u_i)]$ must be lower bounded by integrating $u_i$ below the $\mathbb{E}_i[y^*_i(u_i)]$ quantile with distribution $F_i(\cdot)$, and the last equality holds by replacing variable $u_i$ using $F_i^{-1}(v_i)$. Thus, we have
\begin{equation}\label{PR15}
\begin{aligned}
\text{LP}^{\text{UB}}\leq&\max~~ T\cdot\sum_{i=1}^{n}p_i\cdot r_i\cdot z_i\\
&~\mbox{s.t.}~~~T\cdot\sum_{i=1}^{n}p_i\cdot\int_{0}^{z_i}F_i^{-1}(v_i)dv_i\leq\sum_{j=1}^{m}c_j\\
&~~~~~~~~0\leq z_i\leq 1~~~\forall i
\end{aligned}
\end{equation}
The RHS of \eqref{PR15} is a convex optimization problem over $\{z_i\}$ by noting that for each $i$, if we define $g_i(z)=\int_{0}^{z}F^{-1}_i(v)dv$, then $g_i'(z)=F_i^{-1}(z)$ is a non-decreasing function over $z$, which implies that $g_i(z)$ is a convex function over $z$. Then strong duality implies that
\begin{equation}\label{PR16}
\text{LP}^{\text{UB}}\leq\min_{\mu\geq0}\max_{0\leq z_i\leq 1}L(\mathbf{z},\mu)
\end{equation}
where $L(\mathbf{z},\mu)$ is the lagrangian function defined as follows:
\begin{equation}\label{PR17}
L(\mathbf{z},\mu)=T\cdot\sum_{i=1}^{n}r_i\cdot p_i\cdot z_i+\mu\cdot[\sum_{j=1}^{m}c_i-T\cdot\sum_{i=1}^{n}p_i\cdot\int_{0}^{z_i}F^{-1}_i(v_i)dv_i]
\end{equation}
Denote $\{z^*_i,\mu^*\}$ as the optimal solution of the dual problem, then from the KKT condition, we have that
\begin{equation}\label{PR18}
\frac{\partial L(\mathbf{z}^*,\mu^*)}{\partial z_i}=0\Rightarrow r_i\cdot p_i=\mu^*\cdot p_i\cdot F^{-1}_i(z_i^*)\Rightarrow z_i^*=F_i(\frac{r_i}{\mu^*})
\end{equation}
and
\begin{equation}\label{PR19}
\mu^*\cdot[\sum_{j=1}^{m}c_i-T\cdot\sum_{i=1}^{n}p_i\cdot\int_{0}^{z^*_i}F^{-1}_i(v_i)dv_i]=0
\end{equation}
Then, $\mu^*>0$ implies that
\[
\sum_{j=1}^{m}c_i-T\cdot\sum_{i=1}^{n}p_i\cdot\int_{0}^{z^*_i}F^{-1}_i(v_i)dv_i=0\Rightarrow\sum_{j=1}^{m}c_i-T\cdot\sum_{i=1}^{n}p_i\cdot\int_{0}^{F_i^{-1}(\frac{r_i}{\mu^*})}F^{-1}_i(v_i)dv_i=0
\]
Thus, the optimal dual variable $\mu^*$ is obtained by solving the following equation:
\[
T\cdot\sum_{i=1}^{n}p_i\cdot\int_{0}^{\frac{r_i}{\mu^*}}w_idF_i(w_i)=T\cdot\sum_{i=1}^{n}p_i\cdot\int_{0}^{F_i^{-1}(\frac{r_i}{\mu^*})}F^{-1}_i(v_i)dv_i=\sum_{j=1}^{m}c_i
\]
and $\mu^*=0$ if no $\mu^*$ makes the above equation hold. Denoting $\mu(\mathbf{c})=1/\mu^*$, we finally have $z^*_i=r_i\cdot\mu(\mathbf{c})$ for each $i$ and
\[
\text{LP}^{\text{UB}}\leq\min_{\mu\geq0}\max_{0\leq z_i\leq 1}L(\mathbf{z},\mu)=T\cdot\sum_{i=1}^{n}p_i\cdot r_i\cdot F_i(r_i\cdot\mu(\mathbf{c}))
\]
\Halmos
\end{proof}
When the request arrives and reveals itself as $(r_i, u_i)$, based on the threshold structure of the prophet upper bound \eqref{TR1}, we will make the accept/reject decision by comparing size $u_i$ to the threshold $r_i\cdot \mu$, where we will compute the threshold $\mu$ adaptively by \eqref{TR2}. If the request is accepted, since the reward is resource-independent, it is natural for our policy to assign the resource with the largest capacity to serve it. Our policy is formally presented in Algorithm \ref{ATP3}.
\begin{algorithm}
\begin{algorithmic}[1]
\caption{Adaptive Threshold Policy ($\text{ATP}_2$)}
\label{ATP3}
\STATE At time period $t=1,2,\dots,T$
\STATE Denote the remaining capacity as $\mathbf{c}^t$, and compute the threshold $\mu_t(\mathbf{c}^t)$ by solving the equation
\begin{equation}\label{TR3}
\sum_{i=1}^{n}p_i\cdot\int_{0}^{r_i\cdot\mu_t(\mathbf{c}^t)}w_idF_i(w_i)=\frac{\sum_{j=1}^{m}c^t_j}{T-t+1}
\end{equation}
If there is no threshold $\mu_t(\mathbf{c}^t)$ making the equation holds, we set $\mu_t(\mathbf{c}^t)=+\infty$.
\STATE Denote $j_t$ as the resource with the largest remaining capacity, i.e., $c^t_{j_t}=\max\{ c^t_1,c^t_2,\dots,c^t_m \}$.
\STATE For each $i$, define the truncated threshold $h^i_t(\mathbf{c}^t)=\min\{ c^t_{j_t},r_i\cdot\mu_t(\mathbf{c}^t) \}$.
\STATE After request $t$ arrives and reveals itself as $(r_i,u_i)$, do the following:
\begin{enumerate}[label=(\roman*)]
  \item Assign resource $j_t$ to serve request $t$ if $u_i\leq h^i_t(\mathbf{c}^t)$.
  \item Reject request $t$ if $u_i>h^i_t(\mathbf{c}^t)$.
\end{enumerate}
\end{algorithmic}
\end{algorithm}
Following a similar way of proving Lemma \ref{Main1}, we could obtain the following regret bound of Algorithm \ref{ATP3}.
\begin{theorem}\label{TRMain2}
Suppose Assumption \ref{assumption1} holds and $\hat{K}$ is the constant defined as follows:
\begin{equation}\label{PR20}
\hat{K}=\left\lceil \frac{\sum_{j=1}^{m}c_j}{\sum_{i=1}^{n}p_i\cdot\int_{0}^{r_i\cdot\frac{\bar{\omega}}{r_{\max}}}w_idF_i(w_i)} \right\rceil
\end{equation}
where $r_{\max}=\max\{r_1,r_2,\dots,r_n\}$. Then we have the following regret bound
\begin{equation}\label{PR21}
\begin{aligned}
\text{Regret}(\text{ATP}_2)\leq(\log T+1)\cdot(2mM-m)\cdot r_{\max}+r_{\max}\cdot\hat{K}=O(m\log T)
\end{aligned}
\end{equation}
\end{theorem}
It is also worth-noting that when there is only one resource, then the \textit{typical class} conditions from \cite{arlotto2018logarithmic} is enough to establish the above regret bound. The reason is that in the proof, when there is only one resource, we only need to bound the term $\frac{wF(w)}{\int_{0}^{w}udF(u)}$, which could be upper bounded by a constant $M$ under the \textit{typical class} conditions due to Lemma $4$ in \cite{arlotto2018logarithmic}. Also, as shown in the next theorem, the $\Omega(T^{\frac{1}{1+\alpha}})$ lower bound of $\mathbb{E}[V^{\text{off}}(T)]$ continues to hold. Thus, given the $O(\log T)$ regret bound, Algorithm \ref{ATP3} is asymptotically optimal as $T\rightarrow\infty$.
\begin{theorem}\label{Main4}
Under Assumption \ref{newassump1}, Assumption \ref{assumption1} and Assumption \ref{newassump3}, we have that
\begin{equation}\label{maineq5}
\mathbb{E}[V^{\text{off}}(\mathbf{I})]=\Omega(T^{\frac{1}{1+\alpha}})
\end{equation}
where $\alpha$ is a positive constant satisfying $\frac{1}{\gamma_2}=\lambda^\alpha$. Moreover, the constant terms in $\Omega(\cdot)$ is independent of $n$.
\end{theorem}
Note that both the regret bound in Theorem \ref{TRMain2} and the lower bound of the expected reward obtained by the prophet in Theorem \ref{Main4} depend only on $\hat{K}$ instead of $n$. Thus, Algorithm \ref{ATP3} is asymptotically optimal as $T\rightarrow\infty$ even when the reward distribution has an infinite support.

\section{Concluding Remarks}
In this paper, we consider a general model of the online resource allocation problem and we derive near-optimal policies for our problem under two conditions separately: the reward could only take $n$ values and given each reward realization, the size distribution equals a deterministic weight multiplied by a single-dimensional random variable, or both the reward and the size are independent of which resource assigned. Under each condition, we propose an adaptive threshold policy and prove a $O(\log T)$ regret bound. Our approaches for both conditions are based on showing that the optimal solution to the corresponding prophet upper bound possesses a threshold structure and we could mimick the prophet by following an adaptive threshold policy. Specifically, under the first condition, our policy first divide each resource into $n$ parts and only use one specific part of each resource to serve the requests with the corresponding reward. As a result, the regret bound depends on $n$. We then explore if there exists a policy without dividing the resources and if the regret bound will not depend on the cardinality of the support set of the reward. We show that under the second condition, we could propose another adaptive threshold policy which utilizes each resource as a whole and we show that the corresponding regret bound will be independent of $n$. Our analysis could be directly applied to the case where the cardinality of the support set of the reward is infinitely large and our regret bound continue to hold. Also, we show that under both conditions, the expected reward collected by the prophet is at least $\Omega(T^{\frac{1}{1+\alpha}})$, where $\alpha$ is a positive constant depending on the distribution functions, thus, together with the $O(\log T)$ regret bound, we establish the asymptotic optimality of our policies.

Our future efforts will be devoted to deriving near optimal policies and proving regret bounds for more general settings. For example, the setting where the size distribution could not be expressed explicitly by a single-dimensional random variable, or the setting where the reward distribution has an infinite support and the reward and the size could be resource-dependent.



\ \

\bibliographystyle{ormsv080} 
\bibliography{myreferences} 

\begin{APPENDIX}{Proofs of Lemmas, Propositions and Theorems}

\section{Proof of Lemma \ref{lemma1101}}
We will prove the lemma by showing that both relationships "$\leq$", "$\geq$" hold between $\text{LP}^{\text{UB}}$ and the right hand side of \eqref{up21}. We first prove "$\leq$". Denote $\{x^*_j(\mathbf{r}_i,\mathbf{b}_i\cdot u_i)\}$ as the optimal solution of \eqref{newna1} and define $\hat{c}_{ij}=T\cdot p_i\cdot b_{ij}\cdot \mathbb{E}_i[u_i\cdot x_{j}^*(\mathbf{r}_i,\mathbf{b}_i\cdot u_i)]$. Obviously, the capacity constraints of \eqref{newna1} imply that $\{\hat{c}_{ij}\}$ forms a feasible solution of the RHS of \eqref{up21}. Then we focus on $G_i(\mathbf{\hat{c}}_i,1)$ and denote $\hat{x}_{ij}(u_i)=x^*_j(\mathbf{r}_i,\mathbf{b}_i\cdot u_i)$. It is obvious that $\{\hat{x}_{ij}(u_i)\}$ forms a feasible solution of \eqref{up22}, thus we have that
\[
G_i(\mathbf{\hat{c}}_i,1)\geq T\cdot p_i\cdot \mathbb{E}_i[\sum_{j=1}^{m}r_{ij}\cdot \hat{x}_{ij}(u_i)]=T\cdot p_i\cdot \mathbb{E}_i[\sum_{j=1}^{m}r_{ij}\cdot x^*_{j}(\mathbf{r}_i, \mathbf{b}_i\cdot u_i)]
\]
which implies that
\[
\text{LP}^{\text{UB}}=\sum_{i=1}^{n}T\cdot p_i\cdot \mathbb{E}_i[\sum_{j=1}^{m}r_{ij}\cdot x^*_{j}(\mathbf{r}_i, \mathbf{b}_i\cdot u_i)]\leq \sum_{i=1}^{n}G_i(\mathbf{\hat{c}}_i,1)\leq \text{the~RHS~of~}\eqref{up21}
\]
Next, we prove "$\geq$". Denote $\{c^*_{ij}\}$ as the optimal solution of the RHS of \eqref{up21} and denote $\{x^*_{ij}(u_i)\}$ as the optimal solution of \eqref{up22} for $G_i(\mathbf{c}^*_i,1)$. Then we define $\hat{x}_j(\mathbf{r}_i,\mathbf{b}_i\cdot u_i)=x^*_{ij}(u_i)$ and we have that
\[
0\leq \hat{x}_j(\mathbf{r}_i,\mathbf{b}_i\cdot u_i)\leq 1\text{~~~and~~~}\sum_{j=1}^{m}\hat{x}_j(\mathbf{r}_i,\mathbf{b}_i\cdot u_i)=\sum_{j=1}^{m}x^*_{ij}(u_i)\leq1
\]
Also, we have that
\[
\sum_{i=1}^{n}T\cdot p_i\cdot b_{ij}\cdot\mathbb{E}_i[u_i\cdot \hat{x}_j(\mathbf{r}_i,\mathbf{b}_i\cdot u_i)]=\sum_{i=1}^{n}T\cdot p_i\cdot b_{ij}\cdot\mathbb{E}_i[u_i\cdot x^*_{ij}(u_i)]=\sum_{i=1}^{n}c^*_{ij}\leq c_j
\]
Thus, $\{\hat{x}_j(\mathbf{r}_i,\mathbf{b}_i\cdot u_i)\}$ forms a feasible solution of \eqref{newna1} and we have
\[
\text{LP}^{\text{UB}}\geq \sum_{i=1}^{n}T\cdot p_i\cdot \mathbb{E}_i[\sum_{j=1}^{m}r_{ij}\cdot \hat{x}_{j}(\mathbf{r}_i, \mathbf{b}_i\cdot u_i)]=\sum_{i=1}^{n}T\cdot p_i\cdot \mathbb{E}_i[\sum_{j=1}^{m}r_{ij}\cdot x^*_{ij}(u_i)]=\sum_{i=1}^{n}G_i(\mathbf{c}^*_i,1)=\text{the~RHS~of~}\eqref{up21}
\]
which completes our proof.

\section{Proof of Lemma \ref{Upper1}}
We first show that $G_i(\mathbf{\tilde{c}}_i,t)\leq (T-t+1)\cdot p_i\cdot \sum_{j=1}^{m}(r_{ii_j}-r_{ii_{j+1}})\cdot F_i(\mu^{ii_j}_t(\mathbf{\tilde{c}}_i))$. Suppose $\{x^*_{ii_j}(u_i)\}$ is the optimal solution of \eqref{up22} and we define that $\hat{y}_{ii_j}(u_i)=\sum_{k=1}^{j}x^*_{ii_k}(u_i)$ for each $j$, then it follows that
\[
0\leq \hat{y}_{ii_j}(u_i)\leq 1\text{~~~and~~~}(T-t+1)\cdot p_i\cdot \mathbb{E}_i[u_i\cdot \hat{y}_{ii_j}(u_i)]\leq \sum_{k=1}^{j}\tilde{c}_{ii_k}/b_{ii_k}
\]
Then we define $\hat{z}_{i_j}=\mathbb{E}_i[\hat{y}_{ii_j}(u_i)]$ for each $j$. It is obvious that $0\leq \hat{z}_{i_j}\leq 1$ and we also have the following constraint regarding $\hat{z}_{i_j}$
\begin{equation}\label{up7}
(T-t+1)\cdot p_i\cdot\int_{0}^{F_i^{-1}(\hat{z}_{i_j})}u_idF_i(u_i)\leq (T-t+1)\cdot p_i\cdot \mathbb{E}_i[u_i\cdot \hat{y}_{ii_j}(u_i)]\leq \sum_{k=1}^{j}\tilde{c}_{ii_k}/b_{ii_k}
\end{equation}
The first inequality holds by noting that $\hat{y}_{ii_j}(u_i)\in[0,1]$, then the expectation $\mathbb{E}_i[u_i\cdot \hat{y}_{ii_j}(u_i)]$ must be lower bounded by integrating $u_i$ from the smallest realization $0$ to a realization $a$ such that $P(\{u_i\in[0,a]\})=\mathbb{E}_i[\hat{y}_{ii_j}(u_i)]=\hat{z}_{i_j}$. Also note that
\[\begin{aligned}
G_i(\mathbf{\tilde{c}}_i,t)&=(T-t+1)\cdot p_i\cdot \sum_{j=1}^{m}r_{ii_j}\cdot \mathbb{E}_i[x^*_{ii_j}(u_i)]=(T-t+1)\cdot p_i\cdot \sum_{j=1}^{m}(r_{ii_j}-r_{ii_{j+1}})\cdot\mathbb{E}_i[\hat{y}_{ii_j}(u_i)]\\
&=(T-t+1)\cdot p_i\cdot \sum_{j=1}^{m}(r_{ii_j}-r_{ii_{j+1}})\cdot\hat{z}_{i_j}
\end{aligned}\]
Thus, we have
\begin{equation}\label{up8}
\begin{aligned}
G_i(\mathbf{\tilde{c}}_i,t)\leq&\max~~ (T-t+1)\cdot p_i\cdot\sum_{j=1}^{m}(r_{ii_j}-r_{ii_{j+1}})\cdot z_{i_j}\\
&~\mbox{s.t.}~~~(T-t+1)\cdot p_i\cdot\int_{0}^{F_i^{-1}(z_{i_j})}wdF_i(w)\leq\sum_{k=1}^{j}\tilde{c}_{ii_k}/b_{ii_k}~~~\forall j\\
&~~~~~~~~0\leq z_{i_j}\leq 1~~~\forall j
\end{aligned}
\end{equation}
Note that each constraint in the RHS of \eqref{up8} only involves a single $z_{i_j}$ and the objective function is linear in $z_{i_j}$, thus, the optimization problem is separable over $\{z_{i_j}\}$ and for each $j$, the optimal solution $z_{i_j}^*$ is given by solving the following single-variable optimization problem:
\[
\begin{aligned}
&\max~~ (T-t+1)\cdot p_i\cdot(r_{ii_j}-r_{ii_{j+1}})\cdot z_{i_j}\\
&~\mbox{s.t.}~~~(T-t+1)\cdot p_i\cdot\int_{0}^{F_i^{-1}(z_{i_j})}wdF_i(w)\leq\sum_{k=1}^{j}\tilde{c}_{ii_k}/b_{ii_k}\\
&~~~~~~~~0\leq z_{i_j}\leq 1
\end{aligned}
\]
Since $r_{ii_j}-r_{ii_{j+1}}\geq0$, obviously the optimal solution $z_{i_j}^*$ satisfies
\[
(T-t+1)\cdot p_i\cdot\int_{0}^{F_i^{-1}(z^*_{i_j})}wdF_i(w)=\sum_{k=1}^{j}c_{ii_k}/b_{ii_k},\text{~~if~}(T-t+1)\cdot p_i\cdot\int_{0}^{\infty}wdF_i(w)\geq\sum_{k=1}^{j}\tilde{c}_{ii_k}/b_{ii_k}
\]
and
\[
z^*_{i_j}=1,\text{~~if~}(T-t+1)\cdot p_i\cdot\int_{0}^{\infty}wdF_i(w)<\sum_{k=1}^{j}\tilde{c}_{ii_k}/b_{ii_k}
\]
Thus, we have $\mu^{ii_j}_t(\mathbf{\tilde{c}}_i)=F_i^{-1}(z^*_{i_j})$ and it follows that $G_i(\mathbf{\tilde{c}}_i,t)\leq(T-t+1)\cdot p_i\cdot \sum_{j=1}^{m}(r_{ii_j}-r_{ii_{j+1}})\cdot F_i(\mu^{ii_j}_t(\mathbf{\tilde{c}}_i))$. Note that
\[
(T-t+1)\cdot p_i\cdot \mathbb{E}_i[\sum_{j=1}^{m}r_{ij}\cdot \hat{x}_{ij}(u)]=(T-t+1)\cdot p_i\cdot \sum_{j=1}^{m}(r_{ii_j}-r_{ii_{j+1}})\cdot F_i(\mu^{ii_j}_t(\mathbf{\tilde{c}}_i))
\]
It remains to show that $\{\hat{x}_{ii_j}(u_i)\}$ is feasible to \eqref{up22}. Obviously for each $u$, we have that $\sum_{j=1}^{m}\hat{x}_{ii_j}(u_i)\leq 1$ and for each $j$, we have that
\[
(T-t+1)\cdot p_i\cdot \mathbb{E}_i[u_i\cdot\hat{x}_{ii_j}(u_i)]=(T-t+1)\cdot p_i\cdot\int_{\mu^{ii_{j-1}}_t(\mathbf{\tilde{c}}_i)}^{\mu^{ii_j}_t(\mathbf{\tilde{c}}_i)}u_idF_i(u_i)\leq \tilde{c}_{ii_j}/b_{ii_j}
\]
which completes our proof.

\section{Proof of Lemma \ref{assumplemma1}}
We will first show that under the condition, there exists two constants $M_1, M_2>0$ such that
\begin{equation}\label{appendix1}
M_1\leq \frac{wF(w)}{\int_{0}^{w}udF(u)}\leq M_2~~~\text{for~all~}w\in(0,\bar{\omega})
\end{equation}
The second inequality has already been established in Lemma 4 in \cite{arlotto2018logarithmic} and the proof is replicated as follows for completeness:
\[\begin{aligned}
&0<1-\gamma_1\leq 1-\frac{F(\lambda w)}{F(w)}=\int_{\lambda w}^{w}\frac{dF(u)}{F(w)}~~\text{for~all~}w\in(0,\bar{\omega})\\
&\Rightarrow~~\lambda w(1-\gamma_1)\leq \int_{\lambda w}^{w}\frac{\lambda w}{F(w)}dF(u)\leq\int_{0}^{w}\frac{u}{F(w)}dF(u)\\
&\Rightarrow~~\frac{wF(w)}{\int_{0}^{w}udF(u)}\leq\frac{1}{\lambda(1-\gamma_1)}=M_2\text{~~for~all~}w\in(0,\bar{\omega})
\end{aligned}\]
Thus, we only focus on proving the first inequality. We have that for any $w\in(0,\bar{\omega})$ and any positive integer $k$,
\[
1-\frac{F(\lambda^kw)}{F(\lambda^{k-1}w)}\leq 1-\frac{1}{\gamma_2}\Rightarrow \int_{\lambda^kw}^{\lambda^{k-1}w}\frac{dF(u)}{F(\lambda^{k-1}w)}\leq 1-\frac{1}{\gamma_2}
\]
then we have
\[
\int_{\lambda^kw}^{\lambda^{k-1}w}\frac{udF(u)}{F(\lambda^{k-1}w)}\leq \int_{\lambda^kw}^{\lambda^{k-1}w}\frac{\lambda^{k-1}wdF(u)}{F(\lambda^{k-1}w)}\leq (1-\frac{1}{\gamma_2})\cdot\lambda^{k-1}w
\]
also, notice that $\frac{1}{F(\lambda^{k-1}w)}\geq \gamma_1^{k-1}\cdot\frac{1}{F(w)}$, we have
\[
\int_{\lambda^kw}^{\lambda^{k-1}w}\frac{udF(u)}{F(w)}\leq (1-\frac{1}{\gamma_2})\cdot(\frac{\lambda}{\gamma_1})^{k-1}w
\]
Thus, it holds that
\[
\int_{0}^{w}\frac{udF(u)}{F(w)}=\sum_{k=1}^{+\infty}\int_{\lambda^kw}^{\lambda^{k-1}w}\frac{udF(u)}{F(w)}\leq \sum_{k=1}^{+\infty}(1-\frac{1}{\gamma_2})\cdot(\frac{\lambda}{\gamma_1})^{k-1}w=\frac{w\gamma_1(\gamma_2-1)}{\gamma_1\gamma_2-\lambda\gamma_2}
\]
which implies
\[
\frac{wF(w)}{\int_{0}^{w}udF(u)}\geq\frac{\gamma_1\gamma_2-\lambda\gamma_2}{\gamma_1(\gamma_2-1)}
\]
and completes our proof of \eqref{appendix1}. Moreover, We could also find a positive integer $l$ such that $\frac{M_1}{\lambda^{l-1}}\leq M_2\leq\frac{M_1}{\lambda^l}$, then for any $0<w_1<w_2<\bar{\omega}$, we consider two situations:\\
(i). If $\lambda^l\cdot w_2\leq w_1<w_2$, we have
\[
\frac{w_2(F(w_2)-F(w_1))}{\int_{w_1}^{w_2}udF(u)}\leq \frac{w_2(F(w_2)-F(w_1))}{\int_{w_1}^{w_2}w_1dF(u)}=\frac{w_2}{w_1}\leq 1/\lambda^l
\]
(ii). If $w_1<\lambda^l\cdot w_2$, we have
\[
w_2\cdot F(w_1)\geq\frac{1}{\lambda^l}\cdot w_1\cdot F(w_1)\geq \frac{M_1}{\lambda^l}\cdot \int_{0}^{w_1}udF(u)
\]
Thus, we have that
\[\begin{aligned}
w_2 (F(w_2)-F(w_1))&\leq M_2\cdot\int_{0}^{w_2}udF(u)-\frac{M_1}{\lambda^l}\cdot\int_{0}^{w_1}udF(u)=M_2\cdot\int_{w_1}^{w_2}uF(u)+(M_2-\frac{M_1}{\lambda^l})\cdot\int_{0}^{w_1}udF(u)\\
&\leq M_2\cdot\int_{w_1}^{w_2}udF(u)
\end{aligned}\]
Our proof is completed by setting $M=\max\{ \frac{1}{\lambda^l}, M_2 \}$.

\section{Proof of Lemma \ref{newlemma1}}
Since $F_i(\cdot)$ has a continuous density function $f_i(\cdot)$, implicit function theorem implies that the threshold $\mu^{ii_j}_t(\mathbf{\tilde{c}}_i)$ is differentiable over $\tilde{c}_{ii_k}$ and the derivatives can be obtained by taking derivative over both sides of \eqref{up2}. We have that
\begin{equation}
\frac{\partial \mu^{ii_j}_t(\mathbf{\tilde{c}}_i)}{\partial \tilde{c}_{ii_k}}=\frac{1}{p_i\cdot b_{ii_k}\cdot (T-t+1)\cdot\mu^{ii_j}_t(\mathbf{\tilde{c}}_i)\cdot f_i(\mu^{ii_j}_t(\mathbf{\tilde{c}}_i))}
\end{equation}
Thus, we have that
\[
\frac{\partial}{\partial \tilde{c}_{ii_k}}(\frac{1}{\mu^{ii_j}_t(\mathbf{\tilde{c}}_i)})=-\frac{1}{(\mu^{ii_j}_t(\mathbf{\tilde{c}}_i))^2}\cdot \frac{\partial \mu^{ii_j}_t(\mathbf{\tilde{c}}_i)}{\partial \tilde{c}_{ii_k}}=-\frac{1}{p_i\cdot b_{ii_k}\cdot(T-t+1)\cdot(\mu^{ii_j}_t(\mathbf{\tilde{c}}_i))^3\cdot f_i(\mu^{ii_j}_t(\mathbf{\tilde{c}}_i))}
\]
From Condition $2$ in Assumption \ref{assumption1}, $(\mu^{ii_j}_t(\mathbf{\tilde{c}}_i))^3\cdot f_i(\mu^{ii_j}_t(\mathbf{\tilde{c}}_i))$ is a non-decreasing function over $\mu^{ii_j}_t(\mathbf{\tilde{c}}_i)$. Also, noting that $\mu^{ii_j}_t(\mathbf{\tilde{c}}_i)$ is non-decreasing over $\tilde{c}_{ii_k}$, we have $(\mu^{ii_j}_t(\mathbf{\tilde{c}}_i))^3\cdot f_i(\mu^{ii_j}_t(\mathbf{\tilde{c}}_i))$ is a non-decreasing function over $\tilde{c}_{ii_k}$. Thus, we have that $\frac{\partial}{\partial \tilde{c}_{ii_k}}(\frac{1}{\mu^{ii_j}_t(\mathbf{\tilde{c}}_i)})$ is a non-decreasing function over $\tilde{c}_{ii_k}$, which implies the function $\frac{1}{\mu^{ii_j}_t(\mathbf{\tilde{c}}_i)}$ is a convex function over $\tilde{c}_{ii_k}$.

\section{Proof of Lemma \ref{newlemma2}}
First, note that
\begin{equation}\label{pfap18}
\frac{\partial F_i(\mu^{ii_j}_t(\mathbf{\tilde{c}}_i))}{\partial \tilde{c}_{ii_k}}=f_i(\mu^{ii_j}_{t}(\mathbf{\tilde{c}}_i))\cdot\frac{\partial \mu^{ii_j}_{t}(\mathbf{\tilde{c}}_i)}{\partial \tilde{c}_{ii_k}}=\frac{1}{p_i\cdot b_{ii_k}\cdot(T-t+1)\cdot\mu^{ii_j}_{t}(\mathbf{\tilde{c}}_i)}
\end{equation}
which implies
\begin{equation}\label{pfap19}
p_i\cdot (T-t+1)\cdot(F_i(\mu^{ii_j}_{t}(\mathbf{\tilde{c}}_i))-F_i(\mu^{ii_j}_{t}(\mathbf{\tilde{c}}_i-b_{ii_k}\cdot u\cdot\mathbf{e}_{i_k})))=\int_{\tilde{c}_{ii_k}/b_{ii_k}-u}^{\tilde{c}_{ii_k}/b_{ii_k}}
\frac{1}{\mu^{ii_j}_{t}(\mathbf{\tilde{c}}_i-(\tilde{c}_{ii_k}-b_{ii_k}\cdot v)\cdot\mathbf{e}_{i_k})}dv
\end{equation}
From Lemma \ref{newlemma1}, the function $\frac{1}{\mu^{ii_j}_{t}(\mathbf{\tilde{c}}_i-(\tilde{c}_{ii_k}-b_{ii_k}\cdot v)\cdot\mathbf{e}_{i_k})}$ is a convex function over $v$, which implies that
\begin{equation}\label{pfap20}
\int_{\tilde{c}_{ii_k}/b_{ii_k}-u}^{\tilde{c}_{ii_k}/b_{ii_k}}
\frac{1}{\mu^{ii_j}_{t}(\mathbf{\tilde{c}}_i-(\tilde{c}_{ii_k}-b_{ii_k}\cdot v)\cdot\mathbf{e}_{i_k})}dv
\leq\frac{u}{2}\cdot\left[\frac{1}{\mu^{ii_j}_{t}(\mathbf{\tilde{c}}_i)}+\frac{1}{\mu^{ii_j}_{t}(\mathbf{\tilde{c}}_i-b_{ii_k}\cdot u\cdot\mathbf{e}_{i_k})}\right]
\end{equation}
Thus, we have that
\begin{equation}\label{pfap21}
p_i\cdot(T-t+1)\cdot(F_i(\mu^{ii_j}_{t}(\mathbf{\tilde{c}}_i))-F_i(\mu^{ii_j}_{t}(\mathbf{\tilde{c}}_i-u\cdot\mathbf{e}_{i_k})))\leq \frac{u}{2}\cdot\left[\frac{1}{\mu^{ii_j}_{t}(\mathbf{\tilde{c}}_i)}+\frac{1}{\mu^{ii_j}_{t}(\mathbf{\tilde{c}}_i-b_{ii_k}\cdot u\cdot\mathbf{e}_{i_k})}\right]
\end{equation}
We then define a function $g(u)$ as follows:
\begin{equation}\label{pfap22}
g(u):=\frac{p_i\cdot (T-t+1)}{u^2}\cdot(F_i(\mu^{ii_j}_{t}(\mathbf{\tilde{c}}))-F_i(\mu^{ii_j}_{t}(\mathbf{\tilde{c}}-b_{ii_k}\cdot u\cdot\mathbf{e}_{i_k})))-\frac{1}{u\cdot\mu^{ii_j}_{t}(\mathbf{\tilde{c}}_i)}
\end{equation}
Then the derivative of $g(u)$ over $u$ could be obtained as follows:
\begin{equation}\label{pfap23}
\begin{aligned}
g'(u)&=-\frac{2p_i(T-t+1)}{u^3}\cdot(F_i(\mu^{ii_j}_{t}(\mathbf{\tilde{c}}_i))-F_i(\mu^{ii_j}_{t}(\mathbf{\tilde{c}}_i-b_{ii_k} u\cdot\mathbf{e}_{i_k})))+\frac{1}{u^2}
\cdot\left[\frac{1}{\mu^{ii_j}_{t}(\mathbf{\tilde{c}}_i)}+\frac{1}{\mu^{ii_j}_{t}(\mathbf{\tilde{c}}_i-b_{ii_k}u\cdot\mathbf{e}_{i_k})}\right]\\
&=-\frac{2}{u^3}\left\{ p_i(T-t+1)(F_i(\mu^{ii_j}_{t}(\mathbf{\tilde{c}}_i))-F_i(\mu^{ii_j}_{t}(\mathbf{\tilde{c}}_i-b_{ii_k}u\cdot\mathbf{e}_{i_k})))-\frac{u}{2}\cdot
\left[\frac{1}{\mu^{ii_j}_{t}(\mathbf{\tilde{c}}_i)}+\frac{1}{\mu^{ii_j}_{t}(\mathbf{\tilde{c}}_i-b_{ii_k}u\cdot\mathbf{e}_{i_k})}\right] \right\}
\end{aligned}
\end{equation}
Thus, we have that $g'(u)\geq0$, which implies that $g(u)\leq g(\tilde{c}_{ii_k}/b_{ii_k})$. Then we have
\[\begin{aligned}
&g(u)=\frac{p_i(T-t+1)}{u^2}\cdot(F_i(\mu^{ii_j}_{t}(\mathbf{\tilde{c}}_i))-F_i(\mu^{ii_j}_{t}(\mathbf{\tilde{c}}_i-b_{ii_k}u\cdot\mathbf{e}_{i_k})))-
\frac{1}{u\cdot\mu^{ii_j}_{t}(\mathbf{\tilde{c}}_i)}\\
&\leq g(\tilde{c}_{ii_k}/b_{ii_k})=\frac{p_i(T-t+1)}{(\tilde{c}_{ii_k}/b_{ii_k})^2}\cdot(F_i(\mu^{ii_j}_{t}(\mathbf{\tilde{c}}_i))-F_i(\mu^{ii_j}_{t}(\mathbf{\tilde{c}}_i-\tilde{c}_{ii_k}\cdot\mathbf{e}_{i_k})))-
\frac{1}{\tilde{c}_{ii_k}/b_{ii_k}\cdot\mu^{ii_j}_{t}(\mathbf{\tilde{c}}_i)}\\
\end{aligned}\]
Thus, we have that
\[\begin{aligned}
F_i(\mu^{ii_j}_{t}(\mathbf{\tilde{c}}_i))-F_i(\mu^{ii_j}_{t}(\mathbf{\tilde{c}}_i-b_{ii_k}u\cdot\mathbf{e}_{i_k}))&\leq \frac{u^2}{(\tilde{c}_{ii_k}/b_{ii_k})^2}\cdot(F_i(\mu^{ii_j}_{t}(\mathbf{\tilde{c}}_i))-F_i(\mu^{ii_j}_{t}(\mathbf{\tilde{c}}_i-\tilde{c}_{ii_k}\cdot\mathbf{e}_{i_k})))\\
&+\frac{u}{p_i\cdot(T-t+1)\cdot\mu^{ii_j}_{t}(\mathbf{\tilde{c}}_i)}-\frac{u^2}{\tilde{c}_{ii_k}/b_{ii_k}\cdot p_i(T-t+1)\cdot\mu^{ii_j}_{t}(\mathbf{\tilde{c}}_i)}
\end{aligned}\]
Moreover, due to Lemma \ref{assumplemma1}, we have that
\[\begin{aligned}
\frac{\mu^{ii_j}_{t}(\mathbf{\tilde{c}}_i)\cdot(F_i(\mu^{ii_j}_{t}(\mathbf{\tilde{c}}_i))-F_i(\mu^{ii_j}_{t}(\mathbf{\tilde{c}}_i-\tilde{c}_{ii_k}\cdot\mathbf{e}_{i_k})))   }{\tilde{c}_{ii_k}}&=\frac{\mu^{ii_j}_{t}(\mathbf{\tilde{c}}_i)\cdot(F_i(\mu^{ii_j}_{t}(\mathbf{\tilde{c}}_i))-F_i(\mu^{ii_j}_{t}(\mathbf{\tilde{c}}_i
-\tilde{c}_{ii_k}\cdot\mathbf{e}_{i_k})))   }{p_ib_{ii_k}(T-t+1)\cdot\int_{\mu^{ii_j}_{t}(\mathbf{\tilde{c}}_i-\tilde{c}_{ii_k}\cdot\mathbf{e}_{i_k})}^{\mu^{ii_j}_{t}(\mathbf{\tilde{c}}_i)}wdF_i(w) }\\
&\leq\frac{M}{p_ib_{ii_k}(T-t+1)}
\end{aligned}\]
Thus, we have that
\[
F_i(\mu^{ii_j}_{t}(\mathbf{\tilde{c}}_i))-F_i(\mu^{ii_j}_{t}(\mathbf{\tilde{c}}_i-b_{ii_k}u\cdot\mathbf{e}_{i_k}))
\leq\frac{(M-1)\cdot u^2}{\tilde{c}_{ii_k}/b_{ii_k}\cdot p_i\cdot(T-t+1)\cdot\mu^{ii_j}_{t}(\mathbf{\tilde{c}}_i)}+\frac{u}{p_i\cdot(T-t+1)\cdot\mu^{ii_j}_{t}(\mathbf{\tilde{c}}_i)}
\]
which completes our proof.

\section{Proof of Lemma \ref{Main1}}
Since we focus on a fixed $i$ during the proof, we will omit the index $i$ by using index $j$ to denote index $i_j$, using $\mu^j_t(\mathbf{c}^t)$ to denote $\mu^{ii_j}_t(\mathbf{c}^t_i)$ and using $b_j$ to denote $b_{ii_j}$.\\
We only focus on the time period $t$ such that $t<T-K$. The reason is that for any $t<T-K$, we must have for each $j$ and every possible remaining capacity $\mathbf{c}^t$, the threshold $\mu^j_t(\mathbf{c}^t)$ will satisfy $\mu^j_t(\mathbf{c}^t)<\bar{\omega}$. Thus, by only focusing on $t<T-K$, the value of $\mu^j_t(\mathbf{c}^t)$ is always finite and we could fully exploit the two conditions over the distribution function proposed in Assumption \ref{assumption1}. Since $K$ is a constant independent of $T$, the cost of ignoring the last $K$ time periods is at most a constant regret, which must be upper bounded by $r_{\max}\cdot K$. In what follows, we assume $t<T-K$. Define
\begin{equation}\label{pfap1}
\rho_t(\mathbf{c}^t)=(T-t+1)\cdot p\cdot\sum_{j=1}^{m}(r_j-r_{j+1})\cdot F(\mu^j_t(\mathbf{c}^t))-V_t(\mathbf{c}^t)
\end{equation}
where $V_t(\mathbf{c}^t)$ is the "to-go" expected total reward collected by the policy, i.e., the expected cumulative reward collected by the policy from serving requests with reward $\mathbf{r}_i$ from time period $t$ to the last time period $T$ given the remaining capacity at time period $t$ is $\mathbf{c}^t$. Then our final goal is to bound $\rho_1(\mathbf{c}^*)$. Given the remaining capacity $\mathbf{c}^t$ and the current time period $t$, denote $\{j_1,j_2,\dots,j_W\}$ as a subset of $\{1,2,\dots,m\}$ where $j_1=1$ and $\{j_1,j_2,\dots,j_W\}$ satisfies that
\[
h^1_t(\mathbf{c}^t)=h^{j_1}_t(\mathbf{c}^t)\leq h^{j_2}_t(\mathbf{c}^t)\leq\dots\leq h^{j_W}_t(\mathbf{c}^t)=\max\{ h^1_t(\mathbf{c}^t), h^2_t(\mathbf{c}^t),\dots,h^m_t(\mathbf{c}^t) \}
\]
and for each $1\leq w\leq W$ and every $j$ such that $j_{w}<j\leq j_{w+1}-1$, we must have $h_t^j(\mathbf{c}^t)<h_t^{j_{w}}(\mathbf{c}^t)$, where we set $j_{W+1}=m+1$. Note that the definition of $\{j_1,j_2,\dots,j_W\}$ will enable us to obtain the following relationship.
\begin{claim}
For any $1\leq w\leq W$, we must have
\begin{equation}\label{pfap10}
\mu_t^{j_w}(\mathbf{c}^t)\geq c^t_j/b_j \text{~~for~every~} j\text{~such~that~}j_w<j\leq j_{w+1}-1
\end{equation}
\end{claim}
\begin{proof}{Proof:}
we have that $\mu_t^{j_w}(\mathbf{c}^t)\geq h_t^{j_w}(\mathbf{c}^t)> h_t^j(\mathbf{c}^t)$ for any $j$ such that $j_w<j\leq j_{w+1}-1$. Also, it is obvious that $\mu_t^{j_w}(\mathbf{c}^t)\leq\mu_t^j(\mathbf{c}^t)$ since $j_w<j$. Note that $h_t^j(\mathbf{c}^t)=\min\{c^t_j/b_j,\mu_t^j(\mathbf{c}^t)\}$, it must hold that $\mu_t^{j_w}(\mathbf{c}^t)\geq c^t_j/b_j$.
\Halmos
\end{proof}
~\\
Moreover, according to our policy in Algorithm \ref{ATP1}, for each $1\leq w\leq W$, request $t$ with size $\mathbf{b}\cdot u$ will be served by resource $j_w$ if and only if $h_t^{j_{w-1}}(\mathbf{c}^t)<u\leq h_t^{j_w}(\mathbf{c}^t)$, where we set $j_0=0$ and $h_t^{j_0}(\mathbf{c}^t)=0$, and request $t$ will be discarded if and only if $u>h_t^{j_W}(\mathbf{c}^t)$. Then, we can write down recursion of $V_t(\mathbf{c}^t)$ as follows:
\begin{equation}\label{pfap2}
\begin{aligned}
V_t(\mathbf{c}^t)&=(1-p)V_{t+1}(\mathbf{c}^t)+p\cdot\sum_{w=1}^{W}\int_{h_t^{j_{w-1}}(\mathbf{c}^t)}^{h_t^{j_w}(\mathbf{c}^t)}\{ \tr_{j_w}+V_{t+1}(\mathbf{c}^t-b_{j_w}u\cdot\mathbf{e}_{j_w})\}dF(u)
+p\cdot(1-F(h_t^{j_W}(\mathbf{c}^t)))\cdot V_{t+1}(\mathbf{c}^t)\\
&=(1-p\cdot F(h_t^{j_W}(\mathbf{c}^t)))\cdot V_{t+1}(\mathbf{c}^t)+p\cdot\sum_{w=1}^{W}\int_{h_t^{j_{w-1}}(\mathbf{c}^t)}^{h_t^{j_w}(\mathbf{c}^t)}\{ \tr_{j_w}+V_{t+1}(\mathbf{c}^t-b_{j_w}u\cdot\mathbf{e}_{j_w})\}dF(u)
\end{aligned}
\end{equation}
where $\mathbf{e}_{j_w}$ denotes a $m$ dimensional vector with $1$ at $j_w$th component and $0$ at other components. We could substitute $V_t(\cdot)$ by $(T-t+1)p\cdot\sum_{j=1}^{m}(\tr_j-\tr_{j+1})\cdot F(\mu^j_t(\cdot))-\rho_t(\cdot)$ and substitute $V_{t+1}(\cdot)$ by $(T-t)p\cdot\sum_{j=1}^{m}(\tr_j-\tr_{j+1})\cdot F(\mu^j_{t+1}(\cdot))-\rho_{t+1}(\cdot)$ in the above recursion. Then we will get the recursion of $\rho_t(\mathbf{c}^t)$ as follows:
\begin{equation}
\begin{aligned}
&(T-t+1)p\sum_{j=1}^{m}(\tr_j-\tr_{j+1})F(\mu^j_t(\mathbf{c}^t))-\rho_t(\mathbf{c}^t)=(1-pF(h_t^{j_W}(\mathbf{c}^t))) \{(T-t)p\sum_{j=1}^{m}(\tr_j-\tr_{j+1})F(\mu^j_{t+1}(\mathbf{c}^t))-\rho_{t+1}(\mathbf{c}^t)\}\\
&+p\sum_{w=1}^{W}\int_{h_t^{j_{w-1}}(\mathbf{c}^t)}^{h_t^{j_w}(\mathbf{c}^t)}\{ \tr_{j_w}+(T-t)p\sum_{k=1}^{m}(\tr_k-\tr_{k+1})F(\mu^k_{t+1}(\mathbf{c}^t-b_{j_w}u\cdot\mathbf{e}_{j_w}))-\rho_{t+1}(\mathbf{c}^t-b_{j_w}u\cdot\mathbf{e}_{j_w}) \}dF(u)\\
\end{aligned}
\end{equation}
which implies that
\begin{equation}\label{pfap3}
\begin{aligned}
&\rho_t(\mathbf{c}^t)=(1-pF(h_t^{j_W}(\mathbf{c}^t)))\cdot\rho_{t+1}(\mathbf{c}^t)+p\sum_{w=1}^{W}\int_{h_t^{j_{w-1}}(\mathbf{c}^t)}^{h_t^{j_w}(\mathbf{c}^t)}\rho_{t+1}(\mathbf{c}^t-b_{j_w}u\cdot\mathbf{e}_{j_w})dF(u)\\
&+\underbrace{ p\cdot\sum_{j=1}^{m}(\tr_j-\tr_{j+1})F(\mu^j_t(\mathbf{c}^t))-p\cdot\sum_{w=1}^{W}\tr_{j_w}\cdot(F(h_t^{j_w}(\mathbf{c}^t))-F(h_t^{j_{w-1}}(\mathbf{c}^t)) ) }_{\text{I}}\\
&+\underbrace{  (T-t)p\cdot\sum_{j=1}^{m}(\tr_j-\tr_{j+1})\cdot( F(\mu^j_t(\mathbf{c}^t))-F(\mu^j_{t+1}(\mathbf{c}^t)) )  }_{\text{II}}\\
&+\underbrace{ p\cdot\sum_{w=1}^{W}\int_{h_t^{j_{w-1}}(\mathbf{c}^t)}^{h_t^{j_w}(\mathbf{c}^t)}\{(T-t)p\sum_{k=1}^{m}(\tr_k-\tr_{k+1})\cdot( F(\mu^k_{t+1}(\mathbf{c}^t))-F(\mu^k_{t+1}(\mathbf{c}^t-b_{j_w}u\cdot\mathbf{e}_{j_w})) )\}dF(u)  }_{\text{III}}
\end{aligned}
\end{equation}
Denote $\bar{\rho}_{t+1}$ as the maximal value of $\rho_{t+1}(\mathbf{c}^{t+1})$ for all the possible remaining capacity $\mathbf{c}^{t+1}$. Then we have
\[
\rho_t(\mathbf{c}^t)-\bar{\rho}_{t+1}\leq \text{I}+\text{II}+\text{III}
\]
In the following, we will bound the three terms I, II, III separately.\\
\textbf{Bound I:} By rearranging terms, we get
\begin{equation}\label{pfap9}
\begin{aligned}
\text{I}&=p\cdot\sum_{j=1}^{m}(\tr_j-\tr_{j+1})F(\mu^j_t(\mathbf{c}^t))-p\cdot\sum_{w=1}^{W}\tr_{j_w}\cdot(F(h_t^{j_w}(\mathbf{c}^t))-F(h_t^{j_{w-1}}(\mathbf{c}^t)) )\\
&=p\cdot\sum_{w=1}^{W}\{ \sum_{j=j_w}^{j_{w+1}-1}(\tr_j-\tr_{j+1})\cdot[F(\mu^j_t(\mathbf{c}^t))-F(h_t^{j_w}(\mathbf{c}^t))] \}
\end{aligned}
\end{equation}
Then, for each fixed $j$ such that $j_w\leq j\leq j_{w+1}-1$, we will bound the term $F(\mu^j_t(\mathbf{c}^t))-F(h_t^{j_w}(\mathbf{c}^t))$. We will first prove the following bound.
\begin{claim}
For each $1\leq w\leq W$ and each $j$ such that $j_w\leq j\leq j_{w+1}-1$, we must have
\begin{equation}\label{pfap8}
F(\mu^j_t(\mathbf{c}^t))-F(h_t^{j_w}(\mathbf{c}^t))\leq \frac{M\cdot(j-j_w)}{p\cdot(T-t+1)}+[ F(\mu_t^{j_w}(\mathbf{c}^t))-F(h_t^{j_w}(\mathbf{c}^t)) ]\cdot 1_{\{ c^t_{j_w}/b_{j_w}<\mu_t^{j_w}(\mathbf{c}^t) \}}
\end{equation}
\end{claim}
\begin{proof}{Proof:} note that from Lemma \ref{assumplemma1}, we have that
\begin{equation}\label{pfap6}
\mu_t^j(\mathbf{c}^t)\cdot( F(\mu_t^j(\mathbf{c}^t))-F(\mu_t^{j_w}(\mathbf{c}^t)) )\leq M\cdot\int_{\mu_t^{j_w}(\mathbf{c}^t)}^{\mu_t^j(\mathbf{c}^t)}udF(u)=\frac{M\cdot(c^t_{j_w+1}/b_{j_w+1}+\dots+c^t_j/b_j)}{p\cdot(T-t+1)}
\end{equation}
Thus, we have
\begin{equation}\label{pfap7}
\begin{aligned}
F(\mu^j_t(\mathbf{c}^t))-F(h_t^{j_w}(\mathbf{c}^t))&=F(\mu^j_t(\mathbf{c}^t))-F(\mu_t^{j_w}(\mathbf{c}^t))+F(\mu_t^{j_w}(\mathbf{c}^t))-F(h_t^{j_w}(\mathbf{c}^t))\\
&\leq \frac{M\cdot(c^t_{j_w+1}/b_{j_w+1}+\dots+c^t_j/b_j)}{\mu^j_t(\mathbf{c}^t)\cdot p\cdot(T-t+1)}+F(\mu_t^{j_w}(\mathbf{c}^t))-F(h_t^{j_w}(\mathbf{c}^t))\\
&\leq \frac{M\cdot(j-j_w)}{p\cdot(T-t+1)}+F(\mu_t^{j_w}(\mathbf{c}^t))-F(h_t^{j_w}(\mathbf{c}^t))
\end{aligned}
\end{equation}
where the first inequality holds due to \eqref{pfap6} and the second inequality holds by noting that according to \eqref{pfap10}, we must have
\[
\mu^j_t(\mathbf{c}^t)\geq\mu_t^{j_w}(\mathbf{c}^t)\geq\max\{ c^t_{j_w+1}/b_{j_w+1},\dots,c^t_j/b_j\}
\]
Note that $h_t^{j_w}(\mathbf{c}^t)=\min\{c^t_{j_w}/b_{j_w},\mu_t^{j_w}(\mathbf{c}^t)\}$, then $h_t^{j_w}(\mathbf{c}^t)$ differentiates from $\mu_t^{j_w}(\mathbf{c}^t)$ if and only if $c^t_{j_w}/b_{j_w}<\mu_t^{j_w}(\mathbf{c}^t)$, thus we have
\[
F(\mu^j_t(\mathbf{c}^t))-F(h_t^{j_w}(\mathbf{c}^t))\leq \frac{M\cdot(j-j_w)}{p\cdot(T-t+1)}+[ F(\mu_t^{j_w}(\mathbf{c}^t))-F(h_t^{j_w}(\mathbf{c}^t)) ]\cdot 1_{\{ c^t_{j_w}/b_{j_w}<\mu_t^{j_w}(\mathbf{c}^t) \}}
\]
\Halmos
\end{proof}
We then bound the term $[ F(\mu_t^{j_w}(\mathbf{c}^t))-F(h_t^{j_w}(\mathbf{c}^t)) ]\cdot 1_{\{ c^t_{j_w}/b_{j_w}<\mu_t^{j_w}(\mathbf{c}^t) \}}$.
\begin{claim}
For any $1\leq w\leq W$, it holds that
\begin{equation}\label{pfap12}
[ F(\mu_t^{j_w}(\mathbf{c}^t))-F(h_t^{j_w}(\mathbf{c}^t)) ]\cdot 1_{\{ c^t_{j_w}/b_{j_w}<\mu_t^{j_w}(\mathbf{c}^t) \}}\leq \frac{M\cdot j_w}{p\cdot(T-t+1)}
\end{equation}
\end{claim}
\begin{proof}{Proof:}
we will prove \eqref{pfap12} by induction. For $j_1=1$, we have that
\[\begin{aligned}
[ F(\mu_t^{j_1}(\mathbf{c}^t))-F(h_t^{j_1}(\mathbf{c}^t)) ]\cdot 1_{\{ c^t_{j_1}/b_{j_w}<\mu_t^{j_1}(\mathbf{c}^t) \}}&\leq F(\mu_t^{j_1}(\mathbf{c}^t))\cdot 1_{\{ c^t_{j_1}/b_{j_1}<\mu_t^{j_1}(\mathbf{c}^t) \}}\\
&\leq \frac{M\cdot c^t_{j_1}/b_{j_1}}{\mu_t^{j_1}(\mathbf{c}^t)\cdot p\cdot(T-t+1)}\cdot 1_{\{ c^t_{j_1}/b_{j_1}<\mu_t^{j_1}(\mathbf{c}^t)\}}\\
&\leq \frac{M}{p\cdot(T-t+1)}
\end{aligned}\]
where the second inequality holds due to Lemma \ref{assumplemma1}. Now suppose that for $w-1$, we have
\[
[ F(\mu_t^{j_{w-1}}(\mathbf{c}^t))-F(h_t^{j_{w-1}}(\mathbf{c}^t)) ]\cdot 1_{\{ c^t_{j_{w-1}}/b_{j_{w-1}}<\mu_t^{j_{w-1}}(\mathbf{c}^t) \}}\leq \frac{M\cdot j_{w-1}}{p\cdot(T-t+1)}
\]
Then note that $h_t^{j_{w-1}}(\mathbf{c}^t)\leq h_t^{j_{w}}(\mathbf{c}^t)$, we have
\[\begin{aligned}
&[ F(\mu_t^{j_w}(\mathbf{c}^t))-F(h_t^{j_w}(\mathbf{c}^t)) ]\cdot 1_{\{ c^t_{j_w}/b_{j_w}<\mu_t^{j_w}(\mathbf{c}^t) \}}\\
&\leq [ F(\mu_t^{j_w}(\mathbf{c}^t))- F(\mu_t^{j_{w-1}}(\mathbf{c}^t))+ F(\mu_t^{j_{w-1}}(\mathbf{c}^t))- F(h_t^{j_{w-1}}(\mathbf{c}^t)) ]\cdot 1_{\{ c^t_{j_w}/b_{j_w}<\mu_t^{j_w}(\mathbf{c}^t) \}}\\
&\leq [ F(\mu_t^{j_w}(\mathbf{c}^t))- F(\mu_t^{j_{w-1}}(\mathbf{c}^t)) ]\cdot 1_{\{ c^t_{j_w}/b_{j_w}<\mu_t^{j_w}(\mathbf{c}^t) \}}+F(\mu_t^{j_{w-1}}(\mathbf{c}^t))- F(h_t^{j_{w-1}}(\mathbf{c}^t))\\
&=[ F(\mu_t^{j_w}(\mathbf{c}^t))- F(\mu_t^{j_{w-1}}(\mathbf{c}^t)) ]\cdot 1_{\{ c^t_{j_w}/b_{j_w}<\mu_t^{j_w}(\mathbf{c}^t) \}}+[ F(\mu_t^{j_{w-1}}(\mathbf{c}^t))-F(h_t^{j_{w-1}}(\mathbf{c}^t)) ]\cdot 1_{\{ c^t_{j_{w-1}}/b_{j_{w-1}}<\mu_t^{j_{w-1}}(\mathbf{c}^t) \}}\\
&\leq [ F(\mu_t^{j_w}(\mathbf{c}^t))- F(\mu_t^{j_{w-1}}(\mathbf{c}^t)) ]\cdot 1_{\{ c^t_{j_w}/b_{j_w}<\mu_t^{j_w}(\mathbf{c}^t) \}}+\frac{M\cdot j_{w-1}}{p\cdot(T-t+1)}
\end{aligned}
\]
Note that from Lemma \ref{assumplemma1}, it holds that
\[\begin{aligned}
&[ F(\mu_t^{j_w}(\mathbf{c}^t))- F(\mu_t^{j_{w-1}}(\mathbf{c}^t)) ]\cdot 1_{\{ c^t_{j_w}/b_{j_w}<\mu_t^{j_w}(\mathbf{c}^t) \}}\leq \frac{M\cdot(c^t_{j_{w-1}+1}/b_{j_{w-1}+1}+\dots+c^t_{j_w}/b_{j_w})}{\mu_t^{j_w}(\mathbf{c}^t)\cdot p\cdot(T-t+1)}\cdot 1_{\{ c^t_{j_w}/b_{j_w}<\mu_t^{j_w}(\mathbf{c}^t) \}}\\
&\leq \frac{M\cdot(c^t_{j_{w-1}+1}/b_{j_{w-1}+1}+\dots+c^t_{j_w-1}/b_{j_w-1})}{\mu_t^{j_w}(\mathbf{c}^t)\cdot p\cdot(T-t+1)}+\frac{M\cdot c^t_{j_w}/b_{j_w}}{\mu_t^{j_w}(\mathbf{c}^t)\cdot p\cdot(T-t+1)}\cdot 1_{\{ c^t_{j_w}/b_{j_w}<\mu_t^{j_w}(\mathbf{c}^t) \}}\\
&\leq \frac{M\cdot(c^t_{j_{w-1}+1}/b_{j_{w-1}+1}+\dots+c^t_{j_w-1}/b_{j_w-1})}{\mu_t^{j_w}(\mathbf{c}^t)\cdot p\cdot(T-t+1)}+\frac{M}{p\cdot(T-t+1)}\\
&\leq \frac{M\cdot (j_w-1-j_{w-1})}{p\cdot(T-t+1)}+\frac{M}{p\cdot(T-t+1)}=\frac{M\cdot(j_w-j_{w-1})}{p\cdot(T-t+1)}
\end{aligned}\]
where the last inequality holds by noting that according to \eqref{pfap10}, we must have
\[
\mu_t^{j_w}(\mathbf{c}^t)\geq\mu_t^{j_{w-1}}(\mathbf{c}^t)\geq\max\{ c^t_{j_{w-1}+1}/b_{j_{w-1}+1},\dots,c^t_{j_w-1}/b_{j_w-1}\}
\]
Thus, by induction, for any $1\leq w\leq W$, \eqref{pfap12} holds.
\Halmos
\end{proof}
~\\
Thus from \eqref{pfap8} and \eqref{pfap12}, we have that
\begin{equation}\label{pfap13}
F(\mu^j_t(\mathbf{c}^t))-F(h_t^{j_w}(\mathbf{c}^t))\leq \frac{M\cdot j}{p\cdot(T-t+1)}
\end{equation}
Then by substituting \eqref{pfap13} into \eqref{pfap9}, we can finally upper bound the term I by:
\begin{equation}\label{pfap14}
\text{I}\leq \frac{M}{T-t+1}\cdot\sum_{j=1}^{m}(\tr_j-\tr_{j+1})\cdot j=\frac{M}{T-t+1}\cdot\sum_{j=1}^{m}\tr_j
\end{equation}
We then proceed to bound the term II.\\
\textbf{Bound II}: From the definition of the threshold $\mu^j_t(\mathbf{c}^t)$ in \eqref{up2}, for each $j=1,2,\dots,m$, we have that
\[
\frac{c^t_1/b_1+\dots+c^t_j/b_j}{T-t}-\frac{c^t_1/b_1+\dots+c^t_j/b_j}{T-t+1}=p\cdot\int_{\mu_t^j(\mathbf{c}^t)}^{\mu_{t+1}^j(\mathbf{c}^t)}udF(u)\leq p\cdot\mu^j_{t+1}(\mathbf{c}^t)\cdot[F(\mu^j_{t+1}(\mathbf{c}^t))-F(\mu^j_t(\mathbf{c}^t))]
\]
which implies
\[
(T-t)\cdot p\cdot[F(\mu^j_{t}(\mathbf{c}^t))-F(\mu^j_{t+1}(\mathbf{c}^t))]\leq-\frac{c^t_1/b_1+\dots+c^t_j/b_j}{\mu^j_{t+1}(\mathbf{c}^t)\cdot(T-t+1)}
\]
Thus, we have the following upper bound regarding the term II:
\begin{equation}\label{pfap15}
\text{II}\leq-\sum_{j=1}^{m}(\tr_j-\tr_{j+1})\cdot \frac{c^t_1/b_1+\dots+c^t_j/b_j}{\mu^j_{t+1}(\mathbf{c}^t)\cdot(T-t+1)}
\end{equation}
At last, it remains to bound the term III.\\
\textbf{Bound III:} We begin with an observation that from the definition of the threshold $\mu^j_t(\mathbf{c}^t)$ in \eqref{up2}, if $k<j_w$, then for any $u$, we must have $\mu^k_{t+1}(\mathbf{c}^t)=\mu^k_{t+1}(\mathbf{c}^t-b_{j_w}u\cdot\mathbf{e}_{j_w})$. Thus, we could rewrite the term III as:
\begin{equation}\label{pfap16}
\text{III}=p\cdot\sum_{w=1}^{W}\int_{h_t^{j_{w-1}}(\mathbf{c}^t)}^{h_t^{j_w}(\mathbf{c}^t)}\{(T-t)\cdot p\cdot\sum_{k=j_w}^{m}(\tr_k-\tr_{k+1})\cdot( F(\mu^k_{t+1}(\mathbf{c}^t))-F(\mu^k_{t+1}(\mathbf{c}^t-b_{j_w}u\cdot\mathbf{e}_{j_w})) )\}dF(u)
\end{equation}
Note that from Lemma \ref{newlemma2}, for any $k\geq j_w$, we have that
\[
\begin{aligned}
p(T-t)\cdot(F(\mu^{k}_{t+1}(\mathbf{c}^t))-F(\mu^{k}_{t+1}(\mathbf{c}^t-b_{j_w}u\cdot\mathbf{e}_{j_w}))
\leq\frac{(M-1) u^2}{(c^t_{j_w}/ b_{j_w})\cdot\mu^{k}_{t+1}(\mathbf{c}^t)}+\frac{u}{ \mu^{k}_{t+1}(\mathbf{c}^t)}
\end{aligned}
\]
Then, we have that
\begin{equation}\label{pfap27}
\begin{aligned}
\text{III}&\leq \underbrace{  p(M-1)\cdot\sum_{w=1}^{W}\int_{h_t^{j_{w-1}}(\mathbf{c}^t)}^{h_t^{j_w}(\mathbf{c}^t)}\left\{
\sum_{k=j_w}^{m}(\tr_k-\tr_{k+1})\cdot\frac{u^2}{(c^t_{j_w}/b_{j_w})\cdot \mu^k_{t+1}(\mathbf{c}^t)}
\right\}dF(u)   }_{\text{IV}}\\
&+\underbrace{  p\cdot\sum_{w=1}^{W}\int_{h_t^{j_{w-1}}(\mathbf{c}^t)}^{h_t^{j_w}(\mathbf{c}^t)}\sum_{k=j_w}^{m}(\tr_k-\tr_{k+1})\cdot \frac{u}{\mu^k_{t+1}(\mathbf{c}^t)}dF(u) }_{\text{V}}
\end{aligned}
\end{equation}
We next bound the term IV and V separately. We have that
\begin{equation}\label{pfap28}
\begin{aligned}
\text{IV}&=p(M-1)\sum_{w=1}^{W}\sum_{k=j_w}^{m}\frac{r_k-r_{k+1}}{(c^t_{j_w}/b_{j_w})\cdot \mu^k_{t+1}(\mathbf{c}^t)}\cdot \int_{h_t^{j_{w-1}}(\mathbf{c}^t)}^{h_t^{j_w}(\mathbf{c}^t)}u^2dF(u)\\
&\leq p(M-1)\sum_{w=1}^{W}\sum_{k=j_w}^{m}\frac{r_k-r_{k+1}}{(c^t_{j_w}/b_{j_w})\cdot \mu^k_{t+1}(\mathbf{c}^t)}\cdot h_t^{j_w}(\mathbf{c}^t)\cdot \int_{h_t^{j_{w-1}}(\mathbf{c}^t)}^{h_t^{j_w}(\mathbf{c}^t)}udF(u)\\
&\leq p(M-1)\sum_{w=1}^{W}\sum_{k=j_w}^{m}\frac{r_k-r_{k+1}}{(c^t_{j_w}/b_{j_w}})\cdot \int_{h_t^{j_{w-1}}(\mathbf{c}^t)}^{h_t^{j_w}(\mathbf{c}^t)}udF(u)
\end{aligned}
\end{equation}
where the second inequality holds by noting that when $k\geq j_w$, we must have $h_t^{j_w}(\mathbf{c}^t)\leq\mu_t^{j_w}(\mathbf{c}^t)\leq\mu_t^k(\mathbf{c}^t)\leq\mu_{t+1}^k(\mathbf{c}^t)$. We then bound the term $\frac{1}{c^t_{j_w}/b_{j_w}}\cdot\int_{h_t^{j_{w-1}}(\mathbf{c}^t)}^{h_t^{j_w}(\mathbf{c}^t)}udF(u)$ in \eqref{pfap28}.
\begin{claim}
For any $1\leq w\leq W$, it holds that
\begin{equation}\label{pfap30}
\begin{aligned}
\frac{1}{c^t_{j_w}/b_{j_w}}\cdot\int_{h_t^{j_{w-1}}(\mathbf{c}^t)}^{h_t^{j_w}(\mathbf{c}^t)}udF(u)\leq \frac{j_w}{p(T-t+1)}
\end{aligned}
\end{equation}
\end{claim}
\begin{proof}{Proof:}
For each $j_w$, we define the index $j(w)$:
\begin{equation}\label{pfap29}
j(w):=\max\{ j: 0\leq j<j_w\text{~and~}c_{j}/b_j\geq\mu_t^{j}(\mathbf{c}^t) \}
\end{equation}
where we will denote $c_0=\mu_t^{0}(\mathbf{c}^t)=0$ and such a $j(w)$ always exist. Since $j(w)<j_w$, from the definition of $\{j_1,j_2,\dots,j_W\}$, it must hold that
\[
\mu_t^{j(w)}(\mathbf{c}^t)=h_t^{j(w)}(\mathbf{c}^t)\leq h_t^{j_{w-1}}(\mathbf{c}^t)
\]
then we have
\[\begin{aligned}
\frac{1}{c^t_{j_w}/b_{j_w}}\cdot\int_{h_t^{j_{w-1}}(\mathbf{c}^t)}^{h_t^{j_w}(\mathbf{c}^t)}udF(u)&\leq \frac{1}{c^t_{j_w}/b_{j_w}}\cdot\int_{h_t^{j(w)}(\mathbf{c}^t)}^{h_t^{j_w}(\mathbf{c}^t)}udF(u)
\leq\frac{1}{c^t_{j_w}/b_{j_w}}\cdot\int_{\mu_t^{j(w)}(\mathbf{c}^t)}^{\mu_t^{j_w}(\mathbf{c}^t)}udF(u)\\
&=\frac{1}{c^t_{j_w}/b_{j_w}}\cdot\frac{c^t_{j(w)+1}/b_{j(w)+1}+\dots+c^t_{j_w}/b_{j_w}}{p(T-t+1)}
\end{aligned}\]
where the last equality holds due to the definition of $\mu_t^{j(w)}(\mathbf{c}^t)$ and $\mu_t^{j_w}(\mathbf{c}^t)$ in \eqref{up22}. Moreover, note that from the definition of the index $j(w)$, for any index $q$ such that $j(w)<q<j_w$, it must hold $c^t_q/b_q<\mu^q_t(\mathbf{c}^t)$. Thus, we have that $c^t_q/b_q=h^q_t(\mathbf{c}^t)\leq h_t^{j_w}(\mathbf{c}^t)\leq c^t_{j_w}/b_{j_w}$, which implies that
\[
\frac{c^t_{j(w)+1}/b_{j(w)+1}+\dots+c^t_{j_w}/b_{j_w}}{c^t_{j_w}/b_{j_w}}\leq j_w-j(w)\leq j_w
\]
Thus, the proof is completed.
\Halmos
\end{proof}
By substituting \eqref{pfap30} into \eqref{pfap28}, we have that
\begin{equation}\label{pfap31}
\text{IV}\leq \frac{M-1}{T-t+1}\cdot\sum_{w=1}^{W}\sum_{k=j_w}^{m}(\tr_k-\tr_{k+1})j_w=\frac{M-1}{T-t+1}\cdot\sum_{w=1}^{W}\tr_{j_w}\cdot j_w\leq\frac{M-1}{T-t+1}\cdot\sum_{j=1}^{m}j\cdot \tr_j
\end{equation}
At last, we bound the term V as follows:
\begin{equation}\label{pfap32}
\begin{aligned}
\text{V}&=p\cdot\sum_{w=1}^{W}\int_{h_t^{j_{w-1}}(\mathbf{c}^t)}^{h_t^{j_w}(\mathbf{c}^t)}\sum_{k=j_w}^{m}(\tr_k-\tr_{k+1})\cdot \frac{u}{\mu^k_{t+1}(\mathbf{c}^t)}dF(u)
=p\cdot\sum_{k=1}^{m}\frac{\tr_k-\tr_{k+1}}{\mu^k_{t+1}(\mathbf{c}^t)}\cdot\sum_{j_w:j_w\leq k}\int_{h_t^{j_{w-1}}(\mathbf{c}^t)}^{h_t^{j_w}(\mathbf{c}^t)}udF(u)\\
&\leq p\cdot\sum_{k=1}^{m}\frac{\tr_k-\tr_{k+1}}{\mu^k_{t+1}(\mathbf{c}^t)}\cdot\int_{0}^{\mu^k_t(\mathbf{c}^t)}udF(u)=\sum_{k=1}^{m}(\tr_k-\tr_{k+1})\cdot\frac{c^t_1/b_1+\dots+c^t_k/b_k}{\mu^k_{t+1}(\mathbf{c}^t)\cdot(T-t+1)}
\end{aligned}
\end{equation}
where the first inequality holds by noting that for all $j_w\leq k$, we have $h_t^{j_w}(\mathbf{c}^t)\leq\mu_t^{j_w}(\mathbf{c}^t)\leq\mu_t^k(\mathbf{c}^t)$. Finally, we can bound term III by:
\begin{equation}\label{pfap33}
\text{III}\leq\text{IV}+\text{V}\leq \frac{M-1}{T-t+1}\cdot\sum_{j=1}^{m}j\cdot \tr_j+\sum_{k=1}^{m}(\tr_k-\tr_{k+1})\cdot\frac{c^t_1/b_1+\dots+c^t_k/b_k}{\mu^k_{t+1}(\mathbf{c}^t)\cdot(T-t+1)}
\end{equation}
By combining \eqref{pfap14}, \eqref{pfap15} and \eqref{pfap33}, we have
\begin{equation}\label{pfap34}
\begin{aligned}
\rho_t(\mathbf{c}^t)-\bar{\rho}_{t+1}&\leq \frac{M}{T-t+1}\cdot\sum_{j=1}^{m}\tr_j-\sum_{j=1}^{m}(\tr_j-\tr_{j+1})\cdot \frac{c^t_1/b_1+\dots+c^t_j/b_j}{\mu^j_{t+1}(\mathbf{c}^t)\cdot(T-t+1)}\\
&+\frac{M-1}{T-t+1}\cdot\sum_{j=1}^{m}j\cdot \tr_j+\sum_{k=1}^{m}(\tr_k-\tr_{k+1})\cdot\frac{c^t_1/b_1+\dots+c^t_k/b_k}{\mu^k_{t+1}(\mathbf{c}^t)\cdot(T-t+1)}\\
&=\frac{1}{T-t+1}\cdot \sum_{j=1}^{m}\tr_j\cdot(M(j+1)-j)
\end{aligned}
\end{equation}
holds for any $\mathbf{c}^t$, which implies that
\[
\bar{\rho}_t-\bar{\rho}_{t+1}\leq \frac{1}{T-t+1}\cdot \sum_{j=1}^{m}\tr_j\cdot(M(j+1)-j)
\]
Thus, by noting that $\bar{\rho}_{T-K+1}\leq r_{\max}\cdot K$, we have
\[\begin{aligned}
\rho_1(\mathbf{c}^*)&\leq\bar{\rho}_1\leq\sum_{t=1}^{T-K}(\bar{\rho}_t-\bar{\rho}_{t+1})+\bar{\rho}_{T-K+1}\leq \sum_{t=1}^{T-K}\left(\frac{1}{T-t+1}\right)\cdot \sum_{j=1}^{m}\left(r_j\cdot(M(j+1)-j)\right)+r_{\max}\cdot K\\
&\leq (\log T+1)\cdot\sum_{j=1}^{m}\left(\tr_j\cdot(M(j+1)-j)\right)+r_{\max}\cdot K
\end{aligned}\]
which completes our proof.

\section{Proof of Theorem \ref{Main2}}
Since we have shown
\[
\mathbb{E}[V^{\text{off}}(\mathbf{I})]\geq\mathbb{E}[V^{\text{ATP}_1}(\mathbf{I})]\geq\sum_{i=1}^{n}G_i(\mathbf{c}^*_i,1)-O(\log T)
\]
it is enough to show that $\sum_{i=1}^{n}G_i(\mathbf{c}^*_i,1)=\Omega(T^{\frac{1}{1+\alpha}})$. Also, there must exists a $i$ and $i_j$ such that $p_i>0$, $r_{ii_j}-r_{ii_{j+1}}>0$ and $\sum_{k=1}^{j}c^*_{ii_k}\geq\frac{\sum_{j=1}^{m}c_j}{n}$. If we can show that
\[
T\cdot p_i\cdot (r_{ii_j}-r_{ii_{j+1}})\cdot F_i(\mu^{ii_j}_1(\mathbf{c}^*_i))=\Omega(T^{\frac{1}{1+\alpha}})
\]
then the proof is completed immediately. In what follows, we will use $\mu_T$ to denote the threshold $\mu^{ii_j}_1(\mathbf{c}^*_i)$ and denote $l$ as the positive integer such that $\lambda^l\cdot\bar{\omega}\leq\mu_{T}\leq\lambda^{l-1}\cdot\bar{\omega}$. Then we have
\[
\frac{\sum_{k=1}^{j}c^*_{ii_k}}{\bar{b}\cdot p_i\cdot T}\geq\frac{\sum_{k=1}^{j}c^*_{ii_k}/b_{ii_k}}{p_i\cdot T}=\int_{0}^{\mu_{T}}ud F_i(u)\geq\sum_{k=l}^{+\infty}\int_{\lambda^{k+1}\bar{\omega}}^{\lambda^k\bar{\omega}}ud F_i(u)
\]
where $\bar{b}=\min_{1\leq k\leq j}\{b_{ii_k}\}$. We also have that
\[
\int_{\lambda^{k+1}\bar{\omega}}^{\lambda^k\bar{\omega}}udF_i(u)\geq \lambda^{k+1}\bar{\omega}\cdot[F_i(\lambda^k\bar{\omega})-F_i(\lambda^{k+1}\bar{\omega})]\geq\lambda^{k+1}\bar{\omega}(\gamma_1-1)\cdot F_i(\lambda^{k+1}\bar{\omega})
\geq(\frac{\lambda}{\gamma_2})^{k+1}\bar{\omega}(\gamma_1-1)F_i(\bar{\omega})
\]
The last inequality holds by noting $F_i(\bar{\omega})\leq \gamma_2^{k+1}\cdot F_i(\lambda^{k+1}\bar{\omega})$. Thus, we have that
\[
\frac{\sum_{k=1}^{j}c^*_{ii_k}}{\bar{b}\cdot p_i\cdot T}\geq \bar{\omega}(\gamma_1-1)F_i(\bar{\omega})\cdot \sum_{k=l}^{+\infty}(\frac{\lambda}{\gamma_2})^{k+1}=\bar{\omega}(\gamma_1-1)F_i(\bar{\omega})(\frac{\lambda}{\gamma_2})^{l+1}\cdot\frac{\gamma_2}{\gamma_2-\lambda}
\]
Define a constant $Q=\frac{(\sum_{k=1}^{j}c^*_{ii_k})\cdot(\gamma_2-\lambda)}{\bar{b}\bar{\omega}(\gamma_1-1)\lambda p_i F_i(\bar{\omega})}$, we have that
\[
(\frac{\lambda}{\gamma_2})^l\leq \frac{Q}{T}
\]
Also, since $0<\lambda<1$ and $0<\frac{1}{\gamma_2}<1$, there exists a constant $\alpha>0$ such that $\frac{1}{\gamma_2}=\lambda^\alpha$. Thus, we have
\[
\lambda^{(1+\alpha)l}\leq\frac{Q}{T}\Rightarrow \lambda^l\leq (\frac{Q}{T})^{\frac{1}{1+\alpha}}
\]
Note that from \eqref{appendix1}, we have that
\[
F_i(\mu_{T})\geq\frac{M_1}{\mu_{T}}\cdot\int_{0}^{\mu_{T}}ud F_i(u)=\frac{M_1\cdot (\sum_{k=1}^{j}c^*_{ii_k}/b_{ii_k})}{\mu_{T}\cdot p_i\cdot T}\geq\frac{M_1\lambda(\sum_{k=1}^{j}c^*_{ii_k})}{\hat{b}\cdot\bar{\omega}\cdot p_i\cdot T}\cdot\frac{1}{\lambda^l}
\]
where $\hat{b}=\max_{1\leq k\leq j}\{b_{ii_k}\}$. Thus, we have that
\[\begin{aligned}
T\cdot p_i\cdot (r_{ii_j}-r_{ii_{j+1}})\cdot F_i(\mu_{T})&\geq\frac{M_1\lambda(r_{ii_j}-r_{ii_{j+1}})(\sum_{k=1}^{j}c^*_{ii_k}) }{\hat{b}\cdot\bar{\omega}}\cdot\frac{1}{\lambda^l}\\
&\geq\frac{M_1\lambda^{\frac{2+\alpha}{1+\alpha}}(r_{ii_j}-r_{ii_{j+1}})((\gamma_1-1)\bar{b}p_iF_i(\bar{\omega}))^{\frac{1}{1+\alpha}}}{\hat{b}\bar{\omega}^{\frac{\alpha}{1+\alpha}}(\gamma_2-\lambda)^{\frac{1}{1+\alpha}}}
\cdot(\sum_{k=1}^{j}c^*_{ii_k})^{\frac{\alpha}{1+\alpha}}\cdot T^{\frac{1}{1+\alpha}}\\
&\geq\frac{M_1\lambda^{\frac{2+\alpha}{1+\alpha}}(r_{ii_j}-r_{ii_{j+1}})((\gamma_1-1)\bar{b}p_iF_i(\bar{\omega}))^{\frac{1}{1+\alpha}}(\sum_{j=1}^{m}c_j)^{\frac{\alpha}{1+\alpha}}}{\hat{b}\bar{\omega}^{\frac{\alpha}{1+\alpha}}(\gamma_2-\lambda)^{\frac{1}{1+\alpha}}n^{\frac{\alpha}{1+\alpha}}}
\cdot T^{\frac{1}{1+\alpha}}\\
&=\Omega(T^{\frac{1}{1+\alpha}})
\end{aligned}\]
which completes our proof.

\section{Proof of Theorem \ref{TRMain2}}
We first prove the following lemma regarding the threshold $\mu_t(\mathbf{c}^t)$.
\begin{lemma}\label{aplemma2}
For any $j=1,2,\dots,m$, the function $\frac{1}{\mu_t(\mathbf{c}^t)}$ is a convex function over $c^t_j$.
\end{lemma}
\begin{proof}{Proof:}
Since for each $i$, $F_i(\cdot)$ has a continuous density function $f_i(\cdot)$, then implicit function theorem implies that the threshold $\mu_{t}(\mathbf{c}^t)$ is differentiable over $c^t_{j}$ for each $j$ and the derivatives can be obtained by taking derivative over both sides of \eqref{TR3}. We have that
\[
\frac{\partial \mu_{t}(\mathbf{c}^t)}{\partial c^t_{j}}=\frac{1}{\sum_{i=1}^{n}p_i\cdot(T-t+1)\cdot r_i^2\cdot\mu^j_{t}(\mathbf{c}^t)\cdot f_i(r_i\mu^j_{t}(\mathbf{c}^t))}
\]
Thus, we have the function $\frac{1}{\mu_t(\mathbf{c}^t)}$ is differentiable over $c^t_j$. Note that
\[
\frac{\partial}{\partial c^t_j}(\frac{1}{\mu_t(\mathbf{c}^t)})=-\frac{1}{(\mu_t(\mathbf{c}^t))^2}\cdot \frac{\partial \mu_{t}(\mathbf{c}^t)}{\partial c^t_{j}}=-\frac{1}{\sum_{i=1}^{n}\frac{p_i(T-t+1)}{r_i}\cdot(r_i\mu_t(\mathbf{c}^t))^3\cdot f_i(r_i\mu_t(\mathbf{c}^t))}
\]
From Condition $2$ in Assumption \ref{assumption1}, for each $i$, we have that $(r_i\mu_t(\mathbf{c}^t))^3\cdot f_i(r_i\mu_t(\mathbf{c}^t))$ is a non-decreasing function over $r_i\mu_t(\mathbf{c}^t)$. Also, obviously $r_i\mu_t(\mathbf{c}^t)$ is non-decreasing over $c^t_j$, thus we have $(r_i\mu_t(\mathbf{c}^t))^3\cdot f_i(r_i\mu_t(\mathbf{c}^t))$ is a non-decreasing function over $c^t_j$. Finally, we have that $\frac{\partial}{\partial c^t_j}(\frac{1}{\mu_t(\mathbf{c}^t)})$ is a non-decreasing function over $c^t_j$, which implies the function $\frac{1}{\mu_t(\mathbf{c}^t)}$ is a convex function over $c^t_j$.
\Halmos
\end{proof}
~\\
To derive the regret bound, we study the gap between the expected total reward collected by our policy and the prophet upper bound in Theorem \ref{TRmain1}. Specifically, define
\begin{equation}\label{TRlm1}
\rho_t(\mathbf{c}^t)=(T-t+1)\cdot\sum_{i=1}^{n}p_i\cdot r_i\cdot F_i(r_i\mu_t(\mathbf{c}^t))-V_t(\mathbf{c}^t)
\end{equation}
where $V_t(\mathbf{c}^t)$ is the "to-go" expected total reward collected by the our policy, i.e., the expected cumulative reward collected by our policy from time period $t$ to the last time period $T$ given the remaining capacity at time period $t$ is denoted as $\mathbf{c}^t$. Then we have that
\[
\text{Regret}(\text{ATP}_2)\leq\rho_1(\mathbf{c})
\]
Since $j_t$ is denoted as the resource with largest remaining capacity given the remaining capacity is $\mathbf{c}^t$, i.e., $c^t_{j_t}=\max\{ c^t_1,c^t_2,\dots,c^t_m \}$, and our policy will assign resource $j_t$ to serve request $t$, revealed as $(r_i, u_i)$, as long as $u_i\leq h_t^i(\mathbf{c}^t)=\min\{ c^t_{j_t},r_i\mu_t(\mathbf{c}^t) \}$, we have the following recursion of $V_t(\mathbf{c}^t)$:
\[\begin{aligned}
&V_t(\mathbf{c}^t)=(1-\sum_{i=1}^{n}p_i)V_{t+1}(\mathbf{c}^t)+\sum_{i=1}^{n}p_i\cdot\int_{0}^{h_t^i(\mathbf{c}^t)}\{ r_i+V_{t+1}(\mathbf{c}^t-w_i\cdot\mathbf{e}_{j_t}) \}dF_i(w_i)+\sum_{i=1}^{n}p_i(1-F_i(h_t^i(\mathbf{c}^t)))V_{t+1}(\mathbf{c}^t)\\
&\Rightarrow\\
&V_t(\mathbf{c}^t)=(1-\sum_{i=1}^{n}p_iF_i(h_t^i(\mathbf{c}^t)))V_{t+1}(\mathbf{c}^t)+\sum_{i=1}^{n}p_i\cdot\int_{0}^{h_t^i(\mathbf{c}^t)}\{ r_i+V_{t+1}(\mathbf{c}^t-w_i\cdot\mathbf{e}_{j_t}) \}dF_i(w_i)
\end{aligned}
\]
where $\mathbf{e}_{j_t}$ denotes a $m$ dimensional vector with $1$ at $j_t$th component and $0$ at other components. Then, by substituting \eqref{TRlm1} into the above equation, we get the recursion of $\rho_t(\mathbf{c}^t)$.
\[\begin{aligned}
&(T-t+1)\cdot\sum_{i=1}^{n}p_ir_iF_i(r_i\mu_t(\mathbf{c}^t))-\rho_t(\mathbf{c}^t)=(1-\sum_{i=1}^{n}p_iF_i(h^i_t(\mathbf{c}^t)))\cdot[ (T-t)\sum_{i=1}^{n}p_ir_iF_i(r_i\mu_{t+1}(\mathbf{c}^t))-\rho_{t+1}(\mathbf{c}^t) ]\\
&+\sum_{i=1}^{m}p_i\cdot\int_{0}^{h^i_t(\mathbf{c}^t)}\{ r_i+(T-t)\sum_{k=1}^{n}p_kr_kF_k(r_k\mu_{t+1}(\mathbf{c}^t-w_i\cdot\mathbf{e}_{j_t}))-\rho_{t+1}(\mathbf{c}^t-w_i\cdot\mathbf{e}_{j_t}) \}dF_i(w_i)
\end{aligned}\]
which implies that
\[\begin{aligned}
&\rho_t(\mathbf{c}^t)=(1-\sum_{i=1}^{n}p_iF_i(h^i_t(\mathbf{c}^t)))\cdot\rho_{t+1}(\mathbf{c}^t)+\sum_{i=1}^{n}p_i\cdot\int_{0}^{h^i_t(\mathbf{c}^t)}\rho_{t+1}(\mathbf{c}^t-w_i\cdot\mathbf{e}_{j_t})dF_i(w_i)\\
&+\underbrace{ \sum_{i=1}^{n}p_ir_i\cdot(F_i(r_i\mu_t(\mathbf{c}^t))-F_i(h_t^i(\mathbf{c}^t))  ) }_{\text{I}}+\underbrace{ (T-t)\cdot\sum_{i=1}^{n}p_ir_i\cdot(F_i(r_i\mu_t(\mathbf{c}^t))-F_i(r_i\mu_{t+1}(\mathbf{c}^t))  )  }_{\text{II}}\\
&+\underbrace{\sum_{i=1}^{n}p_i\cdot\int_{0}^{h_t^i(\mathbf{c}^t)}\{ (T-t)\cdot\sum_{k=1}^{n}p_kr_k\cdot(F_k(r_k\mu_{t+1}(\mathbf{c}^t))-F_k(r_k\mu_{t+1}(\mathbf{c}^t-w_i\cdot\mathbf{e}_{j_t}))  )  \}dF_i(w_i)    }_{\text{III}}
\end{aligned}\]
By denoting $\bar{\rho}_{t+1}$ as the maximal value of $\rho_{t+1}(\mathbf{c}^{t+1})$ for all the possible remaining capacity $\mathbf{c}^{t+1}$, we have that
\[
\rho_t(\mathbf{c}^t)-\bar{\rho}_{t+1}\leq\text{I}+\text{II}+\text{III}
\]
In the following, we will bound the term I, II, III separately.\\
\textbf{Bound I:} From Lemma \ref{assumplemma1}, we have that
\[
\frac{r_i\mu_t(\mathbf{c}^t)F_i(r_i\mu_t(\mathbf{c}^t))}{\int_{0}^{r_i\mu_t(\mathbf{c}^t)}w_idF_i(w_i)}\leq M~~~\forall i~~\Rightarrow
\frac{\sum_{i=1}^{n}p_ir_i\mu_t(\mathbf{c}^t)F_i(r_i\mu_t(\mathbf{c}^t))}{\sum_{i=1}^{n}p_i\cdot\int_{0}^{r_i\mu_t(\mathbf{c}^t)}w_idF_i(w_i)}\leq M
\]
Also from the definition of $\mu_t(\mathbf{c}^t)$ in \eqref{TR3}, we have that
\[
\sum_{i=1}^{n}p_ir_i\mu_t(\mathbf{c}^t)F_i(r_i\mu_t(\mathbf{c}^t))\leq\frac{M\cdot(c^t_1+\dots+c^t_m)}{T-t+1}
\]
Thus, note that $h_t^i(\mathbf{c}^t)=\min\{c^t_{j_t},r_i\cdot\mu_t(\mathbf{c}^t)\}$, we have the following bound over the term I:
\[\begin{aligned}
\text{I}&=\sum_{i=1}^{n}p_ir_i\cdot(F_i(r_i\mu_t(\mathbf{c}^t))-F_i(h_t^i(\mathbf{c}^t))  )\leq \sum_{i=1}^{n}p_ir_iF_i(r_i\mu_t(\mathbf{c}^t))\cdot 1_{\{ r_i\cdot\mu_t(\mathbf{c}^t)>c^t_{j_t} \}  }\\
&\leq \sum_{i=1}^{n}p_ir_iF_i(r_i\mu_t(\mathbf{c}^t))\cdot 1_{\{ r_{\max}\cdot\mu_t(\mathbf{c}^t)>c^t_{j_t} \}  }=1_{\{ r_{\max}\cdot\mu_t(\mathbf{c}^t)>c^t_{j_t} \}  }\cdot\sum_{i=1}^{n}p_ir_iF_i(r_i\mu_t(\mathbf{c}^t))\\
&\leq 1_{\{ r_{\max}\cdot\mu_t(\mathbf{c}^t)>c^t_{j_t} \}  }\cdot \frac{M\cdot(c^t_1+\dots+c^t_m)}{\mu_t(\mathbf{c}^t)\cdot(T-t+1)}\leq\frac{M\cdot(c^t_1+\dots+c^t_m)\cdot r_{\max}}{c^t_{j_t}\cdot(T-t+1)}\\
&\leq \frac{M\cdot r_{\max}\cdot m}{T-t+1}
\end{aligned}\]
where $r_{\max}=\{ r_1,r_2,\dots,r_n \}$. In the next, we bound the term II.\\
\textbf{Bound II:} From the definition of $\mu_t(\mathbf{c}^t)$ in \eqref{TR3}, we have
\[
\frac{c^t_1+\dots+c^t_m}{T-t}-\frac{c^t_1+\dots+c^t_m}{T-t+1}=\sum_{i=1}^{n}p_i\cdot\int_{r_i\mu_t(\mathbf{c}^t)}^{r_i\mu_{t+1}(\mathbf{c}^t)}w_idF_i(w_i)\leq\sum_{i=1}^{n}p_ir_i\mu_{t+1}(\mathbf{c}^t)\cdot
(F_i(r_i\mu_{t+1}(\mathbf{c}^t))-F_i(r_i\mu_t(\mathbf{c}^t)))
\]
Thus, we have that
\[
\text{II}=(T-t)\cdot\sum_{i=1}^{n}p_ir_i\cdot(F_i(r_i\mu_t(\mathbf{c}^t))-F_i(r_i\mu_{t+1}(\mathbf{c}^t)))\leq -\frac{c^t_1+\dots+c^t_m}{\mu_{t+1}(\mathbf{c}^t)\cdot(T-t+1)}
\]
At last, it remains to bound III.\\
\textbf{Bound III:} Note that
\[
\frac{\partial \mu_{t+1}(\mathbf{c}^t)}{\partial c^t_{j_t}}=\frac{1}{(T-t)\cdot\sum_{i=1}^{n}p_ir_i^2f_i(r_i\mu_{t+1}(\mathbf{c}^t))\cdot\mu_{t+1}(\mathbf{c}^t)}
\]
Thus, we have that
\[
\frac{\partial}{\partial c^t_{j_t}}((T-t)\cdot\sum_{k=1}^{n}p_kr_kF_k(r_k\mu_{t+1}(\mathbf{c}^t)))=(T-t)\cdot\sum_{k=1}^{n}p_kr_k^2f_k(r_k\mu_{t+1}(\mathbf{c}^t))\cdot\frac{\partial \mu_{t+1}(\mathbf{c}^t)}{c^t_{j_t}}=\frac{1}{\mu_{t+1}(\mathbf{c}^t)}
\]
which implies that
\[
(T-t)\cdot\sum_{k=1}^{n}p_kr_k\cdot(F_k(r_k\mu_{t+1}(\mathbf{c}^t))-F_k(r_k\mu_{t+1}(\mathbf{c}^t-w_i\cdot\mathbf{e}_{j_t}))=\int_{c^t_{j_t}-w_i}^{c^t_{j_t}}\frac{1}{\mu_{t+1}(\mathbf{c}^t(u))}du
\]
where $\mathbf{c}^t(u)=\mathbf{c}^t+(u-c^t_{j_t})\cdot\mathbf{e}_{j_t}$. Also from Lemma \ref{aplemma2}, $1/\mu_{t+1}(\mathbf{c}^t)$ is convex over $c^t_{j_t}$ for any $\mathbf{c}^t$, then we have that:
\[
(T-t)\cdot\sum_{k=1}^{n}p_kr_k\cdot(F_k(r_k\mu_{t+1}(\mathbf{c}^t))-F_k(r_k\mu_{t+1}(\mathbf{c}^t-w_i\cdot\mathbf{e}_{j_t}))\leq\frac{w_i}{2}\cdot\left[ \frac{1}{\mu_{t+1}(\mathbf{c}^t)}+\frac{1}{\mu_{t+1}(\mathbf{c}^t-w_i\cdot\mathbf{e}_{j_t})} \right]
\]
Define a function $g(w_i)$ as
\[
g(w_i):=\frac{T-t}{w_i^2}\cdot\sum_{k=1}^{n}p_kr_k\cdot(F_k(r_k\mu_{t+1}(\mathbf{c}^t))-F_k(r_k\mu_{t+1}(\mathbf{c}^t-w_i\cdot\mathbf{e}_{j_t}))-\frac{1}{w_i\cdot\mu_{t+1}(\mathbf{c}^t)}
\]
then we have that
\[\begin{aligned}
g'(w_i)&=-\frac{2(T-t)}{w_i^3}\cdot\sum_{k=1}^{n}p_kr_k\cdot(F_k(r_k\mu_{t+1}(\mathbf{c}^t))-F_k(r_k\mu_{t+1}(\mathbf{c}^t-w_i\cdot\mathbf{e}_{j_t}))+\frac{1}{w_i^2\cdot\mu_{t+1}(\mathbf{c}^t-w_i\cdot\mathbf{e}_{j_t})}
+\frac{1}{w_i^2\cdot\mu_{t+1}(\mathbf{c}^t)}\\
&=-\frac{2}{w_i^3}\cdot\left[ (T-t)\sum_{k=1}^{n}p_kr_k(F_k(r_k\mu_{t+1}(\mathbf{c}^t))-F_k(r_k\mu_{t+1}(\mathbf{c}^t-w_i\cdot\mathbf{e}_{j_t}))-\frac{w_i}{2}\left(\frac{1}{\mu_{t+1}(\mathbf{c}^t)}+\frac{1}{\mu_{t+1}(\mathbf{c}^t-w_i\cdot\mathbf{e}_{j_t})}\right)  \right]\\
&\geq0
\end{aligned}\]
Thus, $g(w_i)$ is non-decreasing over $w_i$, which implies that, for any $w_i\leq c^t_{j_t}$, we have
\[\begin{aligned}
&g(w_i)=\frac{T-t}{w_i^2}\cdot\sum_{k=1}^{n}p_kr_k\cdot(F_k(r_k\mu_{t+1}(\mathbf{c}^t))-F_k(r_k\mu_{t+1}(\mathbf{c}^t-w_i\cdot\mathbf{e}_{j_t}))-\frac{1}{w_i\cdot\mu_{t+1}(\mathbf{c}^t)}\\
&\leq g(c^t_{j_t})=\frac{T-t}{(c^t_{j_t})^2}\cdot\sum_{k=1}^{n}p_kr_k\cdot(F_k(r_k\mu_{t+1}(\mathbf{c}^t))-F_k(r_k\mu_{t+1}(\mathbf{c}^t-c^t_{j_t}\cdot\mathbf{e}_{j_t}))-\frac{1}{c^t_{j_t}\cdot\mu_{t+1}(\mathbf{c}^t)}\\
\end{aligned}\]
which implies
\begin{equation}\label{20201}
\begin{aligned}
&(T-t)\cdot \sum_{k=1}^{n}p_kr_k\cdot(F_k(r_k\mu_{t+1}(\mathbf{c}^t))-F_k(r_k\mu_{t+1}(\mathbf{c}^t-w_i\cdot\mathbf{e}_{j_t}))\\
&\leq \frac{w_i^2(T-t)}{(c^t_{j_t})^2}\cdot\sum_{k=1}^{n}p_kr_k\cdot(F_k(r_k\mu_{t+1}(\mathbf{c}^t))-F_k(r_k\mu_{t+1}(\mathbf{c}^t-c^t_{j_t}\cdot\mathbf{e}_{j_t}))+\frac{1}{\mu_{t+1}(\mathbf{c}^t)}\cdot[w_i-\frac{w_i^2}{c^t_{j_t}}]
\end{aligned}
\end{equation}
Also note that from Lemma \ref{assumplemma1}, for each $i$, we have that
\[
\frac{r_i\mu_{t+1}(\mathbf{c}^t)\cdot (F_i(r_i\mu_{t+1}(\mathbf{c}^t))-F_i(r_i\mu_{t+1}(\mathbf{c}^t-c^t_{j_t}\cdot\mathbf{e}_{j_t})))}{\int_{r_i\mu_{t+1}(\mathbf{c}^t-c^t_{j_t}\cdot\mathbf{e}_{j_t})}^{r_i\mu_{t+1}(\mathbf{c}^t)}u_idF_i(u_i)}\leq M
\]
which implies
\[\begin{aligned}
\frac{\sum_{i=1}^{n}p_ir_i\mu_{t+1}(\mathbf{c}^t)\cdot (F_i(r_i\mu_{t+1}(\mathbf{c}^t))-F_i(r_i\mu_{t+1}(\mathbf{c}^t-c^t_{j_t}\cdot\mathbf{e}_{j_t})))}{\sum_{i=1}^{n}p_i\cdot\int_{r_i\mu_{t+1}(\mathbf{c}^t-c^t_{j_t}\cdot\mathbf{e}_{j_t})}^{r_i\mu_{t+1}(\mathbf{c}^t)}u_idF_i(u_i)}\leq M
\end{aligned}\]
And from the definition of the threshold $\mu_t(\mathbf{c}^t)$ in \eqref{TR3}, we have
\[
\sum_{i=1}^{n}p_i\cdot\int_{r_i\mu_{t+1}(\mathbf{c}^t-c^t_{j_t}\cdot\mathbf{e}_{j_t})}^{r_i\mu_{t+1}(\mathbf{c}^t)}u_idF_i(u_i)=\frac{c^t_{j_t}}{T-t}
\]
Thus, we have that
\begin{equation}\label{20202}
\sum_{i=1}^{n}p_ir_i\cdot (F_i(r_i\mu_{t+1}(\mathbf{c}^t))-F_i(r_i\mu_{t+1}(\mathbf{c}^t-c^t_{j_t}\cdot\mathbf{e}_{j_t})))\leq\frac{M\cdot c^t_{j_t}}{\mu_{t+1}(\mathbf{c}^t)\cdot(T-t)}
\end{equation}
Substituting \eqref{20202} into \eqref{20201}, we have that
\[\begin{aligned}
&(T-t)\cdot \sum_{k=1}^{n}p_kr_k\cdot(F_k(r_k\mu_{t+1}(\mathbf{c}^t))-F_k(r_k\mu_{t+1}(\mathbf{c}^t-w_i\cdot\mathbf{e}_{j_t}))\\
&\leq\frac{(M-1)\cdot w_i^2}{\mu_{t+1}(\mathbf{c}^t)\cdot c^t_{j_t}}+\frac{w_i}{\mu_{t+1}(\mathbf{c}^t)}
\end{aligned}\]
As a result, we could bound the term III by:
\[\begin{aligned}
\text{III}&=\sum_{i=1}^{n}p_i\cdot\int_{0}^{h_t^i(\mathbf{c}^t)}\{ (T-t)\cdot\sum_{k=1}^{n}p_kr_k\cdot(F_k(r_k\mu_{t+1}(\mathbf{c}^t))-F_k(r_k\mu_{t+1}(\mathbf{c}^t-w_i\cdot\mathbf{e}_{j_t}))  )  \}dF_i(w_i)\\
&\leq \underbrace{  \frac{M-1}{\mu_{t+1}(\mathbf{c}^t)\cdot c^t_{j_t}}\cdot\left[\sum_{i=1}^{n}p_i\cdot\int_{0}^{h_t^i(\mathbf{c}^t)}w_i^2dF_i(w_i) \right] }_{\text{IV}}+\underbrace{ \frac{1}{\mu_{t+1}(\mathbf{c}^t)}\cdot
\left[\sum_{i=1}^{n}p_i\cdot\int_{0}^{h_t^i(\mathbf{c}^t)}w_idF_i(w_i) \right] }_{\text{V}}
\end{aligned}\]
In what follows, we upper bound the term IV and V separately and we then get the upper bound of III. Note that
\[\begin{aligned}
\sum_{i=1}^{n}p_i\cdot\int_{0}^{h_t^i(\mathbf{c}^t)}w_i^2dF_i(w_i)&\leq \sum_{i=1}^{n}p_ih_t^i(\mathbf{c}^t)\cdot\int_{0}^{h_t^i(\mathbf{c}^t)}w_idF_i(w_i)\leq r_{\max}\cdot\mu_t(\mathbf{c}^t)\cdot\sum_{i=1}^{n}p_i\cdot\int_{0}^{h_t^i(\mathbf{c}^t)}w_idF_i(w_i)\\
&\leq r_{\max}\cdot\mu_t(\mathbf{c}^t)\cdot\sum_{i=1}^{n}p_i\cdot\int_{0}^{r_i\mu_t(\mathbf{c}^t)}w_idF_i(w_i)=r_{\max}\cdot\mu_t(\mathbf{c}^t)\cdot\frac{c^t_1+\dots+c^t_m}{T-t+1}\\
&\leq r_{\max}\cdot\mu_{t+1}(\mathbf{c}^t)\cdot\frac{c^t_1+\dots+c^t_m}{T-t+1}
\end{aligned}\]
Thus, we have
\[
\text{IV}=\frac{M-1}{\mu_{t+1}(\mathbf{c}^t)\cdot c^t_{j_t}}\cdot\left[\sum_{i=1}^{n}p_i\cdot\int_{0}^{h_t^i(\mathbf{c}^t)}w_i^2dF_i(w_i) \right]\leq \frac{(M-1)\cdot r_{\max}\cdot(c^t_1+\dots+c^t_m)}{c^t_{j_t}\cdot(T-t+1)}
\leq \frac{(M-1)r_{\max}m}{T-t+1}
\]
It also holds that
\[\begin{aligned}
\text{V}&= \frac{1}{\mu_{t+1}(\mathbf{c}^t)}\cdot\left[\sum_{i=1}^{n}p_i\cdot\int_{0}^{h_t^i(\mathbf{c}^t)}w_idF_i(w_i) \right]\leq  \frac{1}{\mu_{t+1}(\mathbf{c}^t)}\cdot
\left[\sum_{i=1}^{n}p_i\cdot\int_{0}^{r_i\mu_t(\mathbf{c}^t)}w_idF_i(w_i) \right]\\
&=\frac{c^t_1+\dots+c^t_m}{\mu_{t+1}(\mathbf{c}^t)\cdot(T-t+1)}
\end{aligned}\]
Then, we obtain the following upper bound of the term III:
\[
\text{III}\leq \frac{(M-1)r_{\max}m}{T-t+1}+\frac{c^t_1+\dots+c^t_m}{\mu_{t+1}(\mathbf{c}^t)\cdot(T-t+1)}
\]
Thus, we have
\[\begin{aligned}
\rho_t(\mathbf{c}^t)-\bar{\rho}_{t+1}&\leq\text{I}+\text{II}+\text{III}\leq \frac{Mr_{\max}m}{T-t+1}-\frac{c^t_1+\dots+c^t_m}{\mu_{t+1}(\mathbf{c}^t)\cdot(T-t+1)}+\frac{(M-1)r_{\max}m}{T-t+1}+\frac{c^t_1+\dots+c^t_m}{\mu_{t+1}(\mathbf{c}^t)\cdot(T-t+1)}\\
&=\frac{(2mM-m)\cdot r_{\max}}{T-t+1}
\end{aligned}\]
holds for all possible $\mathbf{c}^t$, which implies that
\[
\bar{\rho}_t-\bar{\rho}_{t+1}\leq \frac{(2mM-m)\cdot r_{\max}}{T-t+1}
\]
Moreover, note that $\bar{\rho}_{T-\hat{K}+1}\leq r_{\max}\cdot\hat{K}$, we have that
\[\begin{aligned}
\rho_1(\mathbf{c})&\leq\bar{\rho}_1\leq\sum_{t=1}^{T-\hat{K}}(\bar{\rho}_t-\bar{\rho}_{t+1})+\bar{\rho}_{T-\hat{K}+1}\leq \sum_{t=1}^{T-\hat{K}} \frac{(2mM-m)\cdot r_{\max}}{T-t+1}+r_{\max}\cdot\hat{K}\\
&\leq (\log T+1)\cdot(2mM-m)\cdot r_{\max}+r_{\max}\cdot\hat{K}
\end{aligned}\]
which completes our proof.

\section{Proof of Theorem \ref{Main4}}
Without lose of generality, we normalize $r_{\max}$ to be $1$ and we assume that $T$ is large enough such that for each $i$, $r_i\cdot\mu(\mathbf{c})\leq\bar{\omega}$. Since we have shown in Theorem \ref{Main3} that
\[
\mathbb{E}[V^{\text{off}}(\mathbf{I})]\geq\mathbb{E}[V^{\text{ATP}_1}(\mathbf{I})]\geq T\cdot\sum_{i=1}^{n}p_i\cdot r_i\cdot F_i(r_i\cdot\mu(\mathbf{c}))-O(\log T)
\]
it is enough to show that $T\cdot\sum_{i=1}^{n}p_i r_i F_i(r_i\mu(\mathbf{c}))=\Omega(T^{\frac{1}{1+\alpha}})$. In what follows, we will use $\mu_T$ to denote the threshold $\mu(\mathbf{c})$ in \eqref{TR2} and denote $l$ as the positive integer such that $\lambda^l\cdot\bar{\omega}\leq\mu_{T}\leq\lambda^{l-1}\cdot\bar{\omega}$. We have from \eqref{TR2} that
\[
\frac{\sum_{j=1}^{n}c_j}{T}\geq\sum_{i=1}^{n}p_i\cdot\int_{0}^{r_i\cdot\mu_{T}}u_id F_i(u_i)\geq\sum_{i=1}^{n}p_i\cdot\sum_{k=l}^{+\infty}\int_{r_i\lambda^{k+1}\bar{\omega}}^{r_i\lambda^k\bar{\omega}}u_id F_i(u_i)
\]
We also have that
\[\begin{aligned}
\int_{r_i\lambda^{k+1}\bar{\omega}}^{r_i\lambda^k\bar{\omega}}u_idF_i(u_i)&\geq r_i\lambda^{k+1}\bar{\omega}\cdot[F_i(r_i\lambda^k\bar{\omega})-F_i(r_i\lambda^{k+1}\bar{\omega})]\geq r_i\lambda^{k+1}\bar{\omega}(\gamma_1-1)\cdot F_i(r_i\lambda^{k+1}\bar{\omega})\\
&\geq(\frac{\lambda}{\gamma_2})^{k+1}\bar{\omega}(\gamma_1-1)r_iF_i(r_i\bar{\omega})
\end{aligned}\]
The last inequality holds by noting $F_i(r_i\bar{\omega})\leq \gamma_2^{k+1}\cdot F_i(r_i\lambda^{k+1}\bar{\omega})$. Thus, we have that
\[
\frac{\sum_{j=1}^{m}c_j}{T}\geq \bar{\omega}(\gamma_1-1)\sum_{i=1}^{n}r_ip_iF_i(r_i\bar{\omega})\cdot \sum_{k=l}^{+\infty}(\frac{\lambda}{\gamma_2})^{k+1}=\bar{\omega}(\gamma_1-1)\sum_{i=1}^{n}r_ip_iF_i(r_i\bar{\omega})(\frac{\lambda}{\gamma_2})^{l+1}\cdot\frac{\gamma_2}{\gamma_2-\lambda}
\]
Define a constant $Q=\frac{(\sum_{j=1}^{m}c_j)\cdot(\gamma_2-\lambda)}{\bar{\omega}(\gamma_1-1)\lambda\sum_{i=1}^{n}r_ip_iF_i(r_i\bar{\omega})}$, we have that
\[
(\frac{\lambda}{\gamma_2})^l\leq \frac{Q}{T}
\]
Also, since $0<\lambda<1$ and $0<\frac{1}{\gamma_2}<1$, there exists a constant $\alpha>0$ such that $\frac{1}{\gamma_2}=\lambda^\alpha$. Thus, we have
\[
\lambda^{(1+\alpha)l}\leq\frac{Q}{T}\Rightarrow \lambda^l\leq (\frac{Q}{T})^{\frac{1}{1+\alpha}}
\]
Notice that from \eqref{TR2}, we have
\[
\sum_{i=1}^{n}p_ir_iF_i(r_i\mu_{T})\geq \frac{M_1}{\mu_{T}}\cdot\sum_{i=1}^{n}p_i\int_{0}^{r_i\mu_{T}}u_id F_i(u_i)=\frac{M_1\cdot (\sum_{j=1}^{m}c_j)}{\mu_{T}\cdot T}\geq\frac{M_1\lambda(\sum_{j=1}^{m}c_j)}{\bar{\omega}\cdot T}\cdot\frac{1}{\lambda^l}
\]
Then, we have that
\[\begin{aligned}
T\cdot\sum_{i=1}^{n} p_i r_{i} F_i(r_i\mu_{T})\geq\frac{M_1\lambda(\sum_{j=1}^{m}c_j)}{\bar{\omega}}\cdot\frac{1}{\lambda^l}\geq\frac{M_1\lambda(\sum_{j=1}^{m}c_j)}{Q^{\frac{1}{1+\alpha}}\bar{\omega}}\cdot T^{\frac{1}{1+\alpha}}
\end{aligned}\]
Further note that from \eqref{PR20}, we have
\[
\hat{K}+1\geq \frac{\sum_{j=1}^{m}c_j}{\sum_{i=1}^{n}p_i\cdot\int_{0}^{r_i\cdot\frac{\bar{\omega}}{r_{\max}}}w_idF_i(w_i)}\geq \frac{\sum_{j=1}^{m}c_j}{\sum_{i=1}^{n}p_ir_i\bar{\omega}F_i(r_i\bar{\omega})}
\]
which implies that $Q\leq\frac{(\gamma_2-\lambda)(\hat{K}+1)}{(\gamma_1-1)\lambda}$. Thus we have that
\[
T\cdot\sum_{i=1}^{n} p_i r_{i} F_i(r_i\mu_{T})\geq \frac{M_1\lambda(\sum_{j=1}^{m}c_j)\cdot(\lambda(\gamma_1-1))^{\frac{1}{1+\alpha}}}{\bar{\omega}\cdot((\gamma_2-\lambda)(\hat{K}+1))^{\frac{1}{1+\alpha}}}\cdot T^{\frac{1}{1+\alpha}}=\Omega(T^{\frac{1}{1+\alpha}})
\]
which completes our proof.

\end{APPENDIX}
\end{document}